\documentclass[11pt]{article}
\usepackage[all,2cell]{xy} \UseAllTwocells \SilentMatrices
\usepackage{latexsym,amsfonts,amssymb}
\usepackage{amsmath,amsthm,amscd}
\usepackage{hyperref,psfrag}
\usepackage{graphicx}
\usepackage{epsfig}
\usepackage{cite, tocvsec2}
\usepackage{amsfonts, amstext, mathrsfs}
\usepackage{amsopn,amscd,xspace,enumerate,multicol}
\usepackage{xcolor} 

\usepackage{a4wide}

\normalfont\upshape

\usepackage{fancyhdr}
\newcommand\lie[1]{\mathfrak{#1}}

\newcommand{\n}{\lie{n}}

\newcommand{\g}{\lie{g}}
\newcommand{\s}{\lie{sl}}
\newcommand{\h}{\lie{h}}

\newcommand{\C}{\mathbb{C}}

\newcommand{\PS}{\mathbb{P}}

\newcommand{\Q}{\mathbb{Q}}
\newcommand{\Z}{\mathbb{Z}}

\newcommand{\N}{\mathbb{N}}

\newcommand{\Nt}{{ \widetilde{\mathcal{N}} }}
\newcommand{\Nc}{{\mathcal{N} }}
\newcommand{\Ox}{\mathcal{O}}

\newcommand\op[1]{{\rm{#1}}}
\newcommand\mH{\mathrm{H}}
\newcommand\mHH{\mathrm{HH}}
\newcommand\lra{\longrightarrow}

\newcommand\ad{\mathrm{ad}}

\newcommand{\Ul}{{\mathsf U}}
\newcommand{\Uk}{{\mathcal U}}
\newcommand{\ul}{{\mathsf u}}
\newcommand{\pl}{{\mathsf p}}
\newcommand{\rl}{{\mathsf r}}
\newcommand{\St}{\mathrm{St}}
\newcommand{\HC}{\mathrm{HC}}
\newcommand{\Hig}{\mathrm{Hig}}

\newcommand{\zl}{{\mathsf z }}
\newcommand{\cl}{\mathsf{c}}

\newcommand{\Hom}{\mathrm{Hom}}
\newcommand{\End}{\mathrm{End}}
\newcommand{\Ext}{\mathrm{Ext}}
\newcommand{\Tor}{\mathrm{Tor}}
\newcommand{\dmod}{\mbox{-}\mathrm{mod}}
\newcommand{\qq}{/\!/}

\theoremstyle{theorem}
\newtheorem{theorem}{Theorem}[section]
\newtheorem{corollary}[theorem]{Corollary}
\newtheorem{conjecture}[theorem]{Conjecture}
\newtheorem{lemma}[theorem]{Lemma}
\newtheorem{proposition}[theorem]{Proposition}

\theoremstyle{definition}
\newtheorem{definition}[theorem]{Definition}
\newtheorem{example}[theorem]{Example}
\newtheorem{remark}[theorem]{Remark}
\numberwithin{equation}{section}

\title{Remarks on the derived center of small quantum groups}
\author{Anna Lachowska, You Qi}
\date{April 24, 2021}

\begin{document}

\maketitle

\begin{abstract}
Let $\ul_q(\g)$ be the small quantum group associated with a complex semisimple Lie algebra $\g$ and a primitive root of unity $q$, satisfying certain restrictions. We establish the equivalence between three different actions of $\g$ on the center of $\ul_q(\g)$ and on the higher derived center of $\ul_q(\g)$. Based on the triviality of this action for  $\g = \s_2, \s_3, \s_4$, we conjecture that, in finite type $A$, central elements of the small quantum group $\ul_q(\s_n)$ arise as the restriction of central elements in the big quantum group $\Ul_q(\s_n)$.

We also study the role of an ideal $\zl_{\rm Hig}$ known as the Higman ideal in the center of $\ul_q(\g)$. We show that it coincides with the intersection of the Harish-Chandra center and its Fourier transform, and compute the dimension of $\zl_{\rm Hig}$ in type $A$.  As an illustration we provide a detailed explicit description of the derived center of $\ul_q(\s_2)$ and its various symmetries. 
\end{abstract}

\section{Introduction}
\paragraph{Background.}Let $\g$ be a complex semisimple Lie algebra with the root system $R$ of rank $r$, and $U_v(\g)$ be the Drinfeld-Jimbo quantum enveloping algebra over $\Q(v)$. We will adopt the conventions and definitions of various quantum algebras used in \cite{ABG}. Namely, let $l$ be an odd integer greater than the Coxeter number of $\g$ and prime to the determinant of the Cartan matrix of $\g$, and let $\Ul=\Ul_q(\g)$ be the Lusztig's integral form of $U_v(\g)$ with $v$ specialized at $q$,  a primitive $l$th root of unity. It is usually known as the \emph{ big quantum group } at a root of unity, and contains the divided powers of the generators $E_i^{(n)} = E_i^n/ [n]_{d_i}!$, $F_i^{(n)} = F_i^n/[n]_{d_i}!$, where $d_i \in \{ 1,2,3\}$ symmetrizes the Cartan matrix of $R$. Then $\Ul_q(\g)$ is generated over $\Q[q]$ by $\{ E_{i}, F_{i}, E_{i}^{(l)}, F_{i}^{(l)} \}$, where $i$ indexes simple roots $\alpha_i$ of $\g$,  along with some additional elements of the Cartan subalgebra containing $K_j^{\pm 1}, j = 1, \ldots r$ (see e.g. \cite{LusfdHopf}).   Let $\Uk=\Uk_q(\g)$ be the De Concini-Kac form of $U_v(\g)$ (without the divided powers), also specialized at $q$ equal to a primitive $l$-th root of unity and factored over the ideal $\langle K_i^l -1\rangle_{i=1 , \ldots r}$. Both $\Ul$ and $\Uk$ are Hopf algebras and there exists a Hopf algebra homomorphism 
\[ \Uk \lra \Ul \] 
whose image $\ul=\ul_q(\g)$ is referred to as the \emph{ small quantum group}. Alternatively, $\ul$ can be defined as a subalgebra of $\Ul$ generated by $E_{i}$, $F_{i}$, $K_i^{\pm 1}$, $i = 1 \ldots r$ and factored over the ideal $\langle K_i^l -1\rangle_{i=1 , \ldots r}$.

Lusztig introduced the quantum Frobenius map $\phi$ that extends the assignment $E_{i}^{(l)} \to e_i$, $E_{i} \to 0$, $F_{i}^{(l)} \to f_i$, $F_{i} \to 0$ to an algebra homomorphism from $\Ul$  to $\hat{U}(\g)$, a completion of the universal enveloping algebra $U(\g)$ with respect to the Chevalley generators $\{e_i, f_i \}_{i = 1}^r$. The two-sided ideal $(\ul) \subset \Ul$ is the kernel of $\phi$ that fits into an exact sequence of bi-algebras 
\begin{equation}\label{eqn_quantumFrob}
 0 \lra (\ul) \lra \Ul \stackrel{\phi}{\lra} \hat{U}(\g)  .
\end{equation}
Here $\hat{U}(\g)$ is a completion of $U(\g)$ such that its category of representations can be identified with the category of finite-dimensional representations over the group $G$ of adjoint type, with ${\rm Lie}(G) = \g$. Using the sequence, one can identify the classical universal enveloping algebra $U(\g)$ as the Hopf quotient $\Ul_q(\g)\qq \ul_q(\g)$ (see Definition \ref{def-hopf-quotient}).

A geometric description of the Frobenius sequence given in \cite{ABG, BeLa} has motivated the authors to investigate the center of small quantum groups via algebro-geometric methods \cite{LQ1, LQ2}. The problem of finding an algebraic description of the center of the small quantum group has a long history. The interest in this question is motivated by the similarity with the case of algebraic groups over fields of positive characteristic \cite{Haboush}, by its connection to the representation theory of affine Lie algebras at a negative level \cite{KaLuIandII}, and more recently by potential applications in logarithmic conformal field theories \cite{FGST}. Based on the geometric interpretation given in \cite{ABG, BeLa},  we have developed in \cite{LQ1,LQ2} an algorithmic method that allows us to compute the structure of various blocks of the center as naturally bigraded vector space.  The results obtained in several low-rank cases then led to a conjecture relating the structure of the principal block of the center of $\ul$ in type $A$ to the bigraded vector space of the diagonal coinvariants  \cite{Hai} corresponding to the same root system. 

\paragraph{Methods and results.} In the present work we shift the focus from the center of a given block of $\ul$ (as in \cite{LQ1,LQ2}) to the entire center $\zl(\ul)$. We also naturally extend our study to  the total Hochschild cohomology of $\ul$ (also known as the {\em derived center} of $\ul$). Our main method relies on the classical theorem due to Ginzburg-Kumar \cite{GK}, which establishes a graded algebra isomorphism
\begin{equation}\label{eqn-GK-isomorphism}
\mHH^\bullet (H)\cong \Ext^\bullet_H(\Bbbk, H_\ad):=\mH^\bullet(H,H_\ad).
\end{equation}
for a Hopf algebra $H$. 
Here the left hand side stands for the Hochschild cohomology ring of $H$, while the right hand side denotes the usual Hopf-cohomology ring with coefficients in the left adjoint representation $H_\ad$.  In particular, taking degree zero parts on both sides, the center of $H$ can be identified with the space of $H$-invariants inside $H_\ad$.  

On the other hand, using \eqref{eqn-GK-isomorphism}, there is a split injection $\mH^\bullet(H, \Bbbk) \lra \mH^\bullet(H, H_{\rm ad})$, since $\Bbbk$ is a direct summand in $H_\ad$. It follows that $\mHH^\bullet(H)$ is a module-algebra over the ring $\mH^\bullet(H,\Bbbk)$. When $H=\ul$ defined over $\Bbbk=\C$, Ginzburg and Kumar has exhibited a $\g$-equivariant isomorphism 
\begin{equation}
\mH^\bullet(\ul,\C)\cong \C[\Nc]
\end{equation} 
identifying the cohomology ring with the space of algebraic functions on the nilpotent cone $\Nc$ of $\g$ \cite{GK}. Thus $\mHH^\bullet(\ul)$ is a module over $\C[\Nc]$.

The isomorphism \eqref{eqn-GK-isomorphism} prompts us to investigate further the module structure of $\ul_\ad$. This module carries some distinguished actions:
\begin{itemize}
\item[(i)] the Hopf adjoint action by the big quantum group $\Ul$ on $\Ul_\ad$ that preserves the submodule $\ul_\ad$. This action induces a $U(\g)$-module structure on $\zl(\ul)$ via the quantum Frobenius map \eqref{eqn_quantumFrob};
\item[(ii)] an extended modular group action on $\ul_\ad$ in the sense of \cite{LyuMa}, as $\ul$ is a \emph{ribbon factorizable} Hopf algebra (see Definition \ref{def-Dr-map} and \ref{def-ribbon-Hopf}). This action is first observed in \cite{LMSS,ScWo}. 
\end{itemize} 
These symmetries descend, via the isomorphism \eqref{eqn-GK-isomorphism}, to the level of the (derived) center of the small quantum group.
Using the sequence \eqref{eqn_quantumFrob} and the flatness of $\Ul$ over $\ul$, one may identify
\begin{equation}
\mHH^\bullet(\ul)\cong \Ext_{\ul}^\bullet(\Bbbk,\ul_\ad)\cong \Ext_{\Ul}^\bullet(\Ul\otimes_{\ul}\Bbbk, \ul_\ad)\cong \Ext_{\Ul}^\bullet(U(\g),\ul_{\ad}).
\end{equation}
It follows that the (derived) center of $\ul$ carries a natural $U(\g)$-module structure as well as an extended modular group structure\footnote{The extended modular group action restricts to a projective action of the modular group on the Hochschild cohomology of $\ul$, 
as was noticed in \cite{LMSS, ScWo}.}. One of the main goals of our work is to match this $\g$-action with the geometrically defined $\g$-action on the nilpotent cone $\Nc$ and its Springer resolution $\Nt$ via Schouten brackets (Theorem \ref{thm-Schouten-action}).

 Along the way, we also investigate a remarkable ideal $\zl_{\rm Hig}$, called the \emph{Higman ideal}, in the (derived) center of $\ul$. The ideal  is preserved by both the action of $U(\g)$ and the projective modular group action on $\mHH^\bullet(\ul)$.

The Higman ideal $\zl_{\rm Hig}$ in the (derived) center of $\ul$ has been studied in the framework of Fronenius algebras (e.g.~\cite{ZimRepBook, CoWe}). For a finite dimensional Hopf algebra $H$, it is the space obtained by applying the left integral $\Lambda \in H$ on $H_\ad$: $\zl_{\rm Hig} = {\rm ad} \Lambda (H)$. It is also called the {\it projective center}, since for a finite-dimensional Hopf algebra this ideal is the image of the span of projective characters in the Hopf dual $H^*$ under the \emph{Radford map}. Under the assumption that $H$ is factorizable, it is also the image of the ideal of projective characters under the \emph{ Drinfeld map } (see Lemma \ref{lem-hig-ideal-proj-char} and Propopsition \ref{prop-Higman-in-HC}). 

We prove the following properties of $\zl_{\rm Hig} \subset \mHH^\bullet(\ul)$: 
 \begin{itemize} 
% \item  $\zl_{\rm Hig}$ is  invariant with respect to the projective modular group action, 
  \item[(i)] The ideal  $\zl_{\rm Hig}$ coincides with the intersection of the Harish-Chandra center $\zl_{\HC}\subset \zl(\ul)$ and its Fourier transform: $\zl_{\rm Hig} = \zl_{\HC} \cap {\mathcal F}(\zl_{\HC})$. Its dimension is equal to the number of blocks in $\ul$ (Theorem \ref{thm-hig=intersection}) . In type $A_{n-1}$ this dimension is given by the rational Catalan number $\frac{1}{l+n} {l+n \choose n}$ (Theorem \ref{n_blocks}). 
\item[(ii)] $\zl_{\rm Hig}$ is a homogeneous ideal in the derived center $\mHH^\bullet(\ul)$. Moreover, it is a submodule in $\mHH^\bullet(\ul)$ with respect to the projective $SL(2, \Z)$ action, generated by the operators $\{{\mathcal F} , \mathcal{L}\}$ introduced in \cite{LyuMa}. Here $\mathcal F$ is the quantum Fourier transform, and $\mathcal{L}$ is the multiplication by the ribbon element of $\ul$ (Theorem \ref{hig_ideal_HH}). 
\end{itemize} 
The ideal $\zl_{\rm Hig}$ acquires special significance in the case of the small quantum group, since it carries a  projective action of the modular group,  and corresponds to the space of projective characters $\pl_l$ under an algebra isomorphism  between the left-shifted traces and the center of $\ul$. The ideal  $\pl_l$ has been considered previously in \cite{LaVer}, where it was presented as a non-semisimple analog (with the flavor of ``simple-to-projective duality") of the Verlinde algebra arising from the representation theory of $\Ul_q(\g)$.  We expect  that  $\pl_l$ and $\zl_{\rm Hig}$ will play an important role in the development of the logarithmic field theories, and we hope to study this question in our next paper. 

\paragraph{Summary of sections.}
In Section \ref{sec-Hopf-adjoint}, we present the necessary background material on finite-dimensional Hopf algberas satisfying increasingly stronger requirements, and study their adjoint representations. 
 We introduce  the \emph{quantum Fourier transform} defined by Lyubashenko-Majid \cite{LyuMa}, and studied in \cite{La, CoWe}, and an extended $SL_2(\Z)$-action on the Hopf adjoint module (Theorem \ref{thm-modular-group-action}).  In the remaining part of Section \ref{sec-Hopf-adjoint} we list other useful facts on the Higman ideal for $H$ \cite{ZimRepBook, CoWe}.

In Section \ref{sec-q-group}, we study the actions of the Lie algebra $\g$ and its universal enveloping on the Hochschild cohomology groups of the small quantum group $\ul$, naturally arising in the study of the center: 
\begin{itemize} 
\item The adjoint action of $U(\g)$ on $\zl(\ul)$ via the Frobenius pullback of the $l$-th divided powers of the generators (Theorem \ref{Action-div-powers};) 
\item The action of $U(\g)$ on the total Hochschild cohomology via the Ginzburg-Kumar isomorphism $\mHH^\bullet(\ul) \simeq {\rm Ext}^\bullet(\Ul\qq \ul , \ul_{\rm ad})$  (Corollary \ref{GK-action}). 
\item The natural $\g$-action coming from the geometric interpretation obtained in \cite{BeLa} of blocks of $\mHH^\bullet(\ul)$ as cohomology of certain coherent sheaves on the Springer resolution $\Nt$ associated with $\g$  (Theorem \ref{thm-Schouten-action}). 
\end{itemize} 
We prove that all these actions restrict to the same action of $\g$ on $\zl(\ul)$ (Corollary \ref{GK-action}, Theorem \ref{thm-Schouten-action}). Since the center $\zl(\ul)$ consists entirely of the trivial $\g$-modules in cases $\g =\s_2, \s_3, \s_4$, as exhibited in \cite{LQ1,LQ2}, 
this implies that $\zl(\ul)$ arises as the restriction of the center of the big quantum group in these cases. We conjecture that this property should hold for all $\zl(\ul)$ in type $A$ (Conjecture \ref{triv-g-action}). 

In Section \ref{sec-Hig-ideal}, we study the Higman ideal for the small quantum group $\ul$. The ideal is preserved under the modular group and under the action of $U(\g)$ studied in the previous section (Theorem \ref{thm-hig=intersection} and Proposition \ref{prop-g-preserves-Hig}). In type $A$, the dimension of the Higman ideal is expressible in terms of generalized Catalan number (Theorem \ref{n_blocks}). The Higman ideal admits an interesting ``fusion product'' introduced in \cite{LaVer}, which we intend to study in a follow-up work.

In Section \ref{section_sl2} we give an explicit description of the derived center of $\ul_q(\s_2)$, illustrating the results described above. This is made possible by the previous work of Kerler \cite{Ker}, Feigin-Gainutdinov-Semikhatov-Tipunin \cite{FGST} and Ostrik \cite{Ostrikadjoint}. The case $\g = \s_2$ remains the only case where an explicit description of the adjoint representation $\ul_{\rm ad}$ is available due to \cite{Ostrikadjoint}, and we use it to obtain an explicit description of $\zl(\ul_q(\s_2))$ as a submodule in the adjoint representation (Theorem \ref{thm-sl2-central-elements-in-adrep}).  We further use the Ginzburg-Kumar isomorphism \eqref{eqn-GK-isomorphism} to obtain a detailed description of the entire derived center  as a module algebra over the function space over the nilpotent cone $\C[\Nc]$ (Theorem \ref{thm-HH-sl2}). The algebra structure on the derived center is determined via  the geometric realization of the blocks of the center (Corollary \ref{HH_sl2_alg}). Finally, using the result in \cite{Ker}, we provide a decomposition of the derived center as a projective $SL(2, \Z)$-module in Corollary \ref{HH-sl2-modular}. 

\paragraph{Acknowledgments.} 
 A.~L. would like to thank Azat Gainutdinov for useful discussions about the $\g$-action on the center of a small quantum group and on the derived center of the small quantum $\s_2$. 
Part of this work was carried out during the first author's visit to California Institute of Technology (Caltech). Both authors thank Caltech for its hospitality and support. 
A.~L. is grateful for the support from Facutl\'e des Sciences de Base at \'Ecole Polytechnique F\'ed\'erale de Lausanne.  
Y.~Q.~is partially supported by the National Science Foundation DMS-1947532 while working on the paper.

\section{Hopf-adjoint action and an ideal in  the center of  $\ul$.}\label{sec-Hopf-adjoint}
In this section we collect some basic facts about finite-dimensional Hopf algebras that we will need later. Much of the material is well-known from the works of Drinfeld \cite{DrACHA}, Radford \cite{Radtrace}, Lyubashenko-Majid \cite{LyuMa} and Cohen-Westreich \cite{CoWe}. We prove some of the results for completeness, while streamlining some arguments.

\paragraph{Generalities.}Let $H$ be a Hopf algebra with the structure maps
\[
\Delta: H\lra H\otimes H, \quad \quad S: H\lra H^{\mathrm{op}},\quad \quad \epsilon: H\lra \Bbbk
\] 
known as the \emph{coproduct}, the \emph{antipode} and the \emph{counit}. We will use Sweedler's notation:
\begin{equation}
\Delta(h) =  \sum h_{1}  \otimes h_{2},
\end{equation}
for any $h\in H$. We will also use very often the axiom
\begin{equation}\label{eqn_antipode_axiom}
\sum h_1S(h_2)= \sum S(h_1)h_2 =\epsilon(h)
\end{equation}
which holds for any $h\in H$.

\begin{definition} 
A unital algebra $A$ equipped with a (left) $H$-action
\[
H\times A\lra A,\quad (h,x)\mapsto h\cdot x
\] 
is called an \emph{(left) $H$-module algebra}, if the algebra structure of $A$ is compatible with the $H$-action, in the sense that, for any $h\in H$
\begin{enumerate}
\item[(1)] $h\cdot 1_A=\epsilon(h) 1_A$ where $1_A$ stands for the multiplicative unit element of $A$;
\item[(2)] $h\cdot (xy)=\sum (h_1\cdot x) (h_2\cdot y)$ for all $x,y\in A$.
\end{enumerate}
\end{definition}

Similarly, one can define the right $H$-action version of $H$-module algebras. We leave its definition and analogous properties below as exercise for the reader.

\begin{lemma} \label{H-mod-alg} 
Let $A$ be an $H$-module algebra. Then the $H$-action on $A$ preserves the center $\zl(A)$ of $A$: $H\cdot \zl(A)\subset \zl(A)$.
\end{lemma}
\begin{proof}
Let $h\in H$, $z\in \zl(A)$ and $x\in A$ be arbitrary elements. Then
\begin{align*}
(h\cdot z)x & =\sum (h_1\cdot z )(h_2S(h_3)\cdot x)=\sum h_1\cdot (z(S(h_2)\cdot x))=\sum h_1 \cdot ((S(h_2)\cdot x) z)\\
& =\sum (h_1S(h_2)\cdot x) (h_3\cdot z) =\sum \epsilon(h_1) x (h_3\cdot z)=x (h\cdot z).
\end{align*}
The result follows.
\end{proof}

Given a Hopf algebra  $H$, it is readily seen to be a (left) $H$-module with respect to the {\it (left) Hopf-adjoint action} defined as follows: 
\begin{equation}
H\times H\lra H, \quad (h,a)\mapsto \ad h (a): = \sum h_{1}  a S(h_{2}),
\end{equation}
for all $a, h \in H$.  We will denote this $H$-module by $H_{\mathrm{ad}}$.

Similarly, the \emph{right Hopf-adjoint action} is defined by
\begin{equation}
H\times H\lra H, \quad (a,h)\mapsto  \ad_r h(a): = \sum S(h_{1})  a h_{2},
\end{equation}
for all $a,h\in H$. We will denote this representation by $ H_\ad^\prime $ when needed.

The next result summarizes some well-known properties of the adjoint representation of Hopf algebras that we will use later. We give its proof for the sake of completeness.

\begin{lemma} \label{ad-H-on-itself} 
Let $H$ be a Hopf algebra.
\begin{enumerate}
\item[(i)] The multiplication map of $H$ is an $\ad$-equivariant homomorphism. In particular, $H$ is a module algebra over itself under the adjoint representation.
\item[(ii)] The space of $\ad$-invariants of $H$ is equal to the center $\zl(H)$ of $H$.
\end{enumerate}
\end{lemma}
\begin{proof}
The second statement of (i) follows from its first part, and using that $\ad h (1_H)=\epsilon(h)$ (\ref{eqn_antipode_axiom}). To see the first part, we note that
\[
\ad h (ab)=\sum h_1abS(h_2) = \sum h_1aS(h_2)h_3bS(h_4) = \sum \ad h_1(a) \cdot \ad h_2 (b).
\]
For (ii), it is clear that, if $z\in Z(H)$ is central, then $\ad h (z)=\epsilon(h)z$. Conversely, if $z\in H_\ad$ is  $\ad$-invariant, then
\[
hz=\sum h_1zS(h_2)h_3= \sum \ad h_1(z) h_2= \sum \epsilon(h_1)zh_2=zh
\]
holds for all $h\in H$. The lemma follows.
\end{proof}

\begin{definition}\label{def-hopf-quotient}
Suppose $H$ is a Hopf algebra and $A\subset H$ is a Hopf subalgebra. Then $A$ is called \emph{normal in $H$} if either of the equivalent conditions
\[
\ad h(A)\subset A ,\quad \quad \ad_r h(A)\subset A
\] 
holds for all $h\in H$. When $A$ is normal in $H$, the augmentation ideal $A_+$ of $\epsilon:A\lra \Bbbk$ generates an ideal in $H$ satisfying $A_+H=HA_+$. The quotient algebra 
\[
H\qq A:=\dfrac{H}{HA_+}
\]
inherits a Hopf algebra structure from that of $H$, and is called the \emph{Hopf quotient algebra}.
\end{definition}

\paragraph{Integral in $H$ and the Higman ideal.} We will be interested in finite dimensional Hopf algebras. A particular feature of such Hopf algebras is the existence of a unique left/right integral up to scaling.

\begin{definition}
Let $H$ be a Hopf algebra. An element $\Lambda\in H$ is called a \emph{left integral} if, for any $h\in H$, we have
\[
h\Lambda = \epsilon(h) \Lambda. 
\]
Similarly, a \emph{right integral} is characterized by the condition
\[
\Lambda h= \epsilon(h) \Lambda. 
\]
\end{definition}

A classical result of Sweedler \cite{Sw} says that the space of left integrals in a finite dimensional Hopf algebra is always one-dimensional, and likewise for the right integrals. When a nonzero left integral is also a right integral, it is then a central element in the Hopf algebra, and the Hopf algebra is called \emph{unimodular}.

\begin{corollary}\label{cor-higman-ideal}
Let $H$ be a finite dimensional Hopf algebra. The space $\ad \Lambda (H)$ consists of central elements. Furthermore, $\ad \Lambda (H)$ is an ideal in $\zl(H)$.
\end{corollary}
\begin{proof}
By Lemma \ref{ad-H-on-itself}, it suffices to show that $H$ acts trivially, via the adjoint representation, on this space. Now, for any $x,h\in H$, we have
\[
\ad x  (\ad \Lambda (h))=\ad (x\Lambda)(h)=\ad(\epsilon(x)\Lambda)(h)=\epsilon(x)\ad \Lambda (h).
\]

For the second statement, it suffices to note that, if $z\in \zl(H)$ and $h\in H$, then
\[
z \ad \Lambda (h)=\sum z\Lambda_1 h S(\Lambda_2)=\sum \Lambda_1 zh S(\Lambda_2)=\ad \Lambda(zh)\in \ad\Lambda (H).
\]
 The result follows.
\end{proof}

\begin{definition}\label{def-Higman}
The central ideal $\zl_\Hig(H):=\ad \Lambda(H)$ in $\zl(H)$ is the \emph{Higman ideal} of $H$. 
\end{definition}

\begin{proposition}\label{prop-adLambdaH}
For a finite-dimensional Hopf algebra $H$, a central element $z\in \zl(H)$ belongs to the ideal $\zl_{\Hig}(H)$ if and only if $z$ spans a trivial submodule contained in a projective-injective summand of $H_\ad$ . 
\end{proposition}
\begin{proof}
We will use \cite[Corollary 5.5]{QYHopf} (see also \cite[Section 4]{QYMor} for another proof), which says that a morphism $f:M\lra N$ in the category of $H$-modules factors through a projective-injective module if and only if 
$$f=\Lambda \cdot g =\sum \Lambda_2 g(S^{-1}(\Lambda_1)(\mbox{-}))$$
for some $g\in \Hom_{\Bbbk}(M,N)$.

Now let $M=\Bbbk z$ where $z\in \zl(H)$, and $N=H_\ad$. Then $M\cong \Bbbk$, the trivial $H$-module. The above formula says that the inclusion map $f : M\lra N$ factors through a projective-injective object, necessarily the injective envelope of the trivial module, if and only if there is a $g\in \Hom_{\Bbbk}(\Bbbk,H_\ad)\cong H_\ad$, such that $f=\Lambda\cdot g$. Identifying $f$ and $g$ with their images $z=f(1)\in H_\ad$, $x=g(1)\in H_\ad$, this is equivalent to requiring that 
$$z=\sum \ad \Lambda_2 (g(S^{-1}(\Lambda_1)1))=\ad\Lambda_2 \epsilon(S^{-1}(\Lambda_1)) g(1)=\ad\Lambda(x).$$ 
When the inclusion of $M$ into $H_\ad$ does factor through the injective envelope $I$ of $M$:
\[
\xymatrix{
\Bbbk z \ar[rr]^f \ar[rd] && H_\ad\\
& I\ar[ur]_{f^\prime} &
}
\] 
then $f^\prime$ must be injective as it is nonzero on the socle, and thus $I$ must occur as a summand of $H_\ad$. The result now follows.
%For any algebra $H$, there is a natural identification of the center
%$$\End_{H\otimes H^{\mathrm{op}}}(H) \cong \zl(H), \quad \quad f\mapsto f(1). $$ 
%
%Now, fix $H$ to be a finite dimensional Hopf algebra. Then $H$ is Frobenius, and the class of injective modules coincide with the class of projectives. Under the above identification, the Higman ideal can be identified with the image of those bimodule maps that factor through a projective $H\otimes H^{\mathrm{op}}$-module (see, for instance, \cite[Proposition 5.9.9]{ZimRepBook} or \cite{QYMor}). Upon restriction of the natural algebra embedding $\delta: H\lra H\otimes H^{\mathrm{op}}$, $h\mapsto \sum h_1\otimes S(h_2)$, $H\otimes H^{\mathrm{op}}$ restricts to a projective $H$-module by pull-back along $\delta$, while $H$ restricts to $H_\ad$. The first statement follows.
%
%For the last statement, note that any one-dimensional subspace, say $L\subset \zl(H)$, spans a copy of the trivial representation by Lemma \ref{ad-H-on-itself}. Then, the factorization above, when restricted to $L$, necessarily goes through a map $g:I\lra H_\ad$, where $I$ is the (projective) injective envelope of the trivial module. Therefore, $I$ must occur as a summand in $H_\ad$ containing $L$, since $f$ restricted to the socle of $I$ is nonzero. The result follows. 
\end{proof}

 For more details on this ideal for Frobenius algebras, we refer the reader to \cite[Section 5.9]{ZimRepBook}. It has been also intensively studied  for Hopf algebras (see, for instance, \cite{CoWe}). The next proposition gives a useful dimension count of the Higman ideal. Although we will not need the result, it will help put some later discussions  in a more general context.

\begin{proposition}\label{prop-dim-proj-center}
If $H$ is a finite-dimensional unimodular Hopf algebra, then
\[
\mathrm{dim}_\Bbbk(\zl_\Hig(H))=\mathrm{rank}(C_H\otimes_\Z\Bbbk),
\]
where $C_H$ stands for the Cartan decomposition matrix of projective $H$-modules in terms of simples.
\end{proposition}
\begin{proof}
This follows from \cite[Corollary 5.9.15]{ZimRepBook} since when $H$ is finite-dimensional, it is a Frobenius algebra. 
\end{proof}

\paragraph{The Radford isomorphism.} 
When $H$ is a unimodular Hopf algebra, Radford \cite{Radtrace} introduced an isomorphism between the space of certain shifted trace-like linear functionals on $H$ with the center of $H$. 

\begin{definition}
Define the commutative subalgebras $\cl(H),\cl_l(H), \cl_r(H)\subset H^*$ by
\[
\cl(H)=\{f\in H^*|f(ab)=f(ba),~\forall a, b \in H\},
\]
\[
\cl_l(H):=\{f\in H^*|f(ab)=f(bS^{2}(a)),~\forall a,b\in H\},
\]
\[
\cl_r(H):=\{f\in H^*|f(ab)=f(bS^{-2}(a)),~\forall a, b\in H\}.
\]
They will be referred to as the algebras of \emph{trace-like functionals}, \emph{left shifted trace-like functionals} and \emph{right-shifted trace-like functionals} respectively.
\end{definition}

\begin{remark}\label{rmk-S-on-clcr}
It is easily verified that, if $f\in \cl_l(H)$ then we have $f\circ S \in \cl_r(H)$. Likewise $f\in \cl_r(H)$ implies $f\circ S \in \cl_l(H)$. Thus precomposing with $S$ defines an isomorphism between $\cl_l(H)$ and $\cl_r(H)$, and precomposing with $S^2$ induces automorphisms of $\cl_l(H)$ and $\cl_r(H)$ respectively. When $S^2$ is inner, the last automorphisms are trivial. 
\end{remark}

\begin{example}\label{eg-shifted-traces}
\begin{enumerate}
\item[(1)] The counit map $\epsilon:H\lra \Bbbk$ is clearly an element in all $\cl$, $\cl_r$ and $\cl_l$.
\item[(2)] Given any $H$-module $V$, its \emph{character} $\chi_V$, defined by
\[
\chi_V(h):=\mathrm{Tr}\big|_V(h),
\]
is clearly an element of $\cl(H)$. Therefore we have a natural subalgebra of $\cl(H)$
\begin{subequations}\label{eqn-characters}
\begin{equation}
 \rl(H):=G_0(H\dmod)\otimes_\Z\Bbbk ,
\end{equation}
where $G_0(H\dmod)$ is the Grothendieck ring spanned by the characters of finite-dimensional $H$-modules. 
The algebra structure on $\cl(H)$ corresponds to tensor product of $H$-modules, and $\epsilon=\chi_\Bbbk$ is the unit of the multiplication. 

We also define
\begin{equation}
\pl(H):=\{\chi_P\in \rl(H)|P:~\textrm{projective $H$-module}\}\otimes_\Z\Bbbk
\end{equation}
 \end{subequations}
Then $\pl(H)$ is an ideal in $\cl(H)$, since the tensor product of a projective $H$-module with any $H$-module remains projective. Furthermore, the ideal $\pl(H)$ in $\rl(H)$ has dimension equal to the rank of $C_H\otimes_\Z \Bbbk$, where $C_H$ is the Cartan matrix of $H$.
\item[(3)] It is known that \cite{Radtrace}, when $H$ is unimodular, the left and right integrals of the dual Hopf algebra $H^*$, denoted $\lambda_l$ and $\lambda_r$ respectively, belong to $\cl_l(H)$ and $\cl_r(H)$ respectively. These elements are characterized by the properties that
\[
f\lambda_l =f(1)\lambda_l,\quad \quad \lambda_rf=\lambda_rf(1)
\]
for all $f\in H^*$.
\end{enumerate}
\end{example}

The following property is the analogue of Lemma \ref{ad-H-on-itself} 

\begin{proposition}\label{prop-left-trace-functionals}
A linear functional $f\in H^*$ belongs to $\cl_l$ if and only if it defines a morphism of $H$-modules $f:H_\ad\lra \Bbbk$. In other words, there is an equality
\begin{equation}
\cl_l  = \Hom_H(H_\ad,\Bbbk).
\end{equation}
\end{proposition}
\begin{proof}
Suppose $f\in \cl_l$. For any $h,x\in H$, we have
\begin{align*}
f(\ad h (x)) & = \sum f( h_1 x S(h_2)) = \sum f(S^{-1}(h_2)h_1x) = \epsilon(h) f(x).
\end{align*}
Conversely, given $f:H_\ad\lra \Bbbk$, we have the equality
\begin{align*}
f(hx) & = \sum f(h_1xS(h_2S(h_3))) = \sum f(h_1 x S^2(h_3)S(h_2))=\sum f(\ad h_1 (xS^{2}(h_2)))\\
& = \epsilon(h_1)f(xS^2(h_2))=f(xS^2(h)),
\end{align*}
which holds for any $h,x\in H$. The result follows.
\end{proof}

Using the right adjoint action, we may also identify
\begin{equation}
\cl_r(H) \cong \Hom_H( H_\ad^\prime, \Bbbk).
\end{equation}

\begin{corollary}\label{cor-trace-kills-nilpotents}
If $z\in \zl(H)$ is contained in the radical of the Hopf-adjoint module $H_\ad$, then $f(z)=0$ for all $f\in \cl_l(H)$. In particular, if $z$ is in the Higman ideal $\zl_\Hig(H)$ for a non-semisimple $H$, then $z$ is annihilated by all shifted trace-like functionals in $\cl_l(H)$.
\end{corollary}
\begin{proof}
The above proposition shows that, as $\Bbbk$ is a simple $H$-module,
\[
\cl_l(H)=\Hom_H(H_\ad, \Bbbk)\cong \Hom_{H}(H_\ad/\mathrm{Rad}(H_\ad),\Bbbk),
\]
where $\mathrm{Rad}(M)$ denotes the  radical of an $H$-module $M$. For the second statement, it suffices to note that, when $H$ is nonsemisimple, $\Lambda$ spans a two-sided ideal of $H$ which is nilpotent by the uniqueness of integrals up to rescaling. It follows that $\Lambda\in \mathrm{Rad}(H)$, and $\ad \Lambda (H) \subset \mathrm{Rad}(H_\ad)$.
\end{proof}

\begin{remark}
The converse of the second statement in the corollary does not hold, as can be seen from the example small quantum $\s_2$ in Theorem \ref{thm-sl2-central-elements-in-adrep}.
\end{remark}

\begin{theorem}\label{thmRadfordiso} 
\begin{enumerate}
\item[(i)] There are isomorphism of $H$-modules
\[
\psi_l: H \lra H^*, \quad h\mapsto \lambda_l((\mbox{-})h), 
\quad \quad \quad
\psi_r: H\lra H^*, \quad h\mapsto \lambda_r(S(h)(\mbox{-}))
\]
\item[(ii)] Upon restriction of the above isomorphisms to the center, there are isomorphisms of modules over $\zl(H)$
\[
\psi_l: \zl(H) \stackrel{\cong}{\lra} \cl_l(H),\quad z\mapsto \lambda_l(z(\mbox{-})),
\quad \quad \quad 
\psi_r: \zl(H) \stackrel{\cong}{\lra} \cl_r(H),\quad z\mapsto \lambda_r(S(z)(\mbox{-})).
\]
\end{enumerate}
\end{theorem}
\begin{proof}
This is the main theorem of \cite{Radtrace}.
\end{proof}  

The isomorphisms in Theorem \ref{thmRadfordiso} are usually known as the \emph{Radford isomorphims}. In this work, we will mostly reserve this term for $\psi_l$. We also obtain, as a corollary, a simpler proof of \cite[Theorem 2.9]{CoWe} under a slightly weaker hypothesis.

\begin{corollary}
 For a unimodular $H$, $\Lambda$ lies in $ \zl_\Hig(H)$ if and only if $H$ is semisimple.
\end{corollary}
\begin{proof}
For a general unimodular $H$, we may always normalize $\lambda_l(\Lambda)=1$. Thus if $\Lambda=\ad \Lambda(h)$ for some $h\in H$, then
\[
\lambda_l(\Lambda)=\lambda_l(\ad \Lambda (h))=\epsilon(\Lambda)\lambda_l(h),
\]
showing that $\epsilon(\Lambda)\neq 0$. Thus $H$ is semisimple. Conversely, if $H$ is semisimple, then $\epsilon(\Lambda)\neq 0$. Using that $\Lambda$ is both a left and right integral, we have 
$$ 
\ad \Lambda (\Lambda ) = \sum \Lambda_1 \Lambda S(\Lambda_2)= \sum \epsilon(\Lambda_1) \Lambda S(\Lambda_2)
= \epsilon(S(\Lambda)) \Lambda=\epsilon(\Lambda) \Lambda,
$$
 so that $\Lambda=\ad \Lambda\left(\frac{\Lambda}{\epsilon(\Lambda)}\right)$. The result follows.
\end{proof}

Another immediate consequence is the following.

\begin{corollary} 
Let $H$ be a non-semisimple finite dimensional Hopf algebra.  Then $\psi_l(z)\big|_{\zl(H)}\equiv 0$ for any  $z \in \zl_{\Hig}(H)$.
\end{corollary} 

\begin{proof} 
If $z=\ad \Lambda (x)$  and $y\in \zl(H)$, we have
\[ 
\psi_l ( \ad \Lambda (x))(y)  = \sum \lambda_l ( \Lambda_1 x S(\Lambda_2) y) = \sum \lambda_l ( \Lambda_1 x yS(\Lambda_2)) = \varepsilon(\Lambda) \lambda_l (xy) =0 . \qedhere
\] 
\end{proof} 

Let us point out the connection of Proposition \ref{prop-left-trace-functionals} with the Radford isomorphism. To do this, consider the \emph{left dual adjoint action} of $H$ on $H^*$ defined by
\[
H\otimes H^*\lra H^*, \quad (h,f)\mapsto \ad^* h(f),
\]
where $\ad^* h(f) (x)=f(\ad S^{-1}(h)(x))$ for any $x\in H$. We denote $H^*$ with this left $H$-module structure by $H^*_\ad$.

\begin{lemma}\label{lem-Rad-ad-inv}
\begin{enumerate}
\item[(i)] The space of $\ad^*$-invariants in $H$ is equal to $\cl_l(H)$.
\item[(ii)] The Radford isomorphism is an $H$-intertwining map $\psi_l:H_\ad\lra H^*_\ad$.
\end{enumerate}
\end{lemma}
\begin{proof}
Since $S^{-1}$ is a bijection of $H$ and $\epsilon\circ S^{-1}=\epsilon$, the first statement is a direct corollary of Proposition \ref{prop-left-trace-functionals}.

The second statement follows from the easy computation:
\begin{align*}
\psi_l(\ad h(x)) & = \sum \lambda_l((\mbox{-})h_1 x S(h_2))=\sum \lambda_l(S^{-1}(h_2)(\mbox{-})h_1x ) \\
 & = \sum \lambda_l(S^{-1}(h)_1 (\mbox{-}) S(S^{-1}(h)_2)x)=\sum\lambda_l (\ad S^{-1}( h)(\mbox{-})x)\\
  & = \ad^* h (\lambda_l((\mbox{-})x)) = \ad^*h (\psi_l(x)). \qedhere
\end{align*}
\end{proof}

\paragraph{The Drinfeld isomorphism.} 
We start this part by recalling some basic notions of \cite{DrACHA}. A (finite-dimensional) Hopf algebra $H$ is called \emph{quasi-triangular} if there exists an invertible element
$R = \sum R_1\otimes R_2 \in H \otimes H$, such that
\begin{enumerate}
\item[(i)] For any $x\in H$, we have $ \Delta^{op}(x) = R \Delta(x) R^{-1} $;
\item[(ii)]$(\Delta\otimes \mathrm{Id})(R)=R_{13}R_{23}$ and $(\mathrm{Id}\otimes \Delta)(R)=R_{13}R_{12}$.
\end{enumerate}
Here we adopt the usual convention that $R_{13}=\sum R_1\otimes 1\otimes R_2\in H^{\otimes 3}$ etc.

Denote by $u=\sum S(R_2)R_1$, where $R_{21}=\sum R_2\otimes R_1\in H\otimes H$.  Then $S^2(h)=uhu^{-1}$ for any $h\in H$. In this case, $S^2$ acts trivially on $\zl(H)$, $\cl_l(H)$ and $\cl_r(H)$ (see Remark \ref{rmk-S-on-clcr}). Furthermore, there are isomorphisms of commutative algebras 
\begin{subequations}
\begin{align}
\mu_l:  \cl(H) \cong \cl_l(H), & \quad f(\mbox{-})\mapsto f(u(\mbox{-})),\\
\mu_r:  \cl(H)\cong \cl_r(H),  & \quad f(\mbox{-})\mapsto f(u^{-1}(\mbox{-})).
\end{align}
\end{subequations}
Under these isomorphisms, the natural subspaces of (projective) characters \eqref{eqn-characters} give rise to shifted characters:
\begin{subequations}\label{eqn-shifted-characters}
\begin{equation}
\rl_l(H):=\mu_l(\rl(H)) ,\quad \quad \pl_l(H):= \mu_l(\pl(H)),
\end{equation}
\begin{equation}
\rl_r(H):=\mu_r(\rl(H)) ,\quad \quad \pl_r(H):= \mu_r(\pl(H)).
\end{equation}
\end{subequations}

The following  result (without the left shifting) is more generally true for symmetric Frobenius algebras (see, for instance, \cite[Proposition 2.1]{CoWe}).

\begin{lemma}\label{lem-hig-ideal-proj-char}
Let $H$ be a quasi-triangular Hopf algebra, then
\[
\psi_l(\zl_{\Hig}(H))=\pl_l(H).
\]
\end{lemma}

\begin{definition} \label{def-Dr-map}
Let $H$ be a (finite-dimensional) quasi-triangular Hopf algebra. Define the \emph{Drinfeld map} $\jmath_r : H^* \lra H$ by 
\[ 
\jmath_r(f) = m(f \otimes {\rm id}) (R_{21}R),
\] 
where $m:H\otimes H\lra H$ is the multiplication map.

A quasi-triangular Hopf algebra $H$ is called \emph{factorizable} if $\jmath_r$ is surjective.
\end{definition}

When $H$ is factorizable,  the Drinfeld map restricts to an isomorphism of commutative algebras $\jmath_r : \cl_r(H) \lra \zl(H)$ \cite{DrACHA}.  %This isomorphism will be also be called the {\em Drinfeld isomorphism}. 
To translate the result to left shifted trace-like functionals $\cl_l(H)$, let us define the \emph{left shifted Drinfeld isomorphism} to be
\begin{equation}\label{eqn-left-shifted-Drinfeld-map}
\jmath_l:=\jmath_r\circ S^{-1}: H^*\lra H, \quad \quad f\mapsto m(f\circ S^{-1} \otimes \mathrm{Id})(R_{21}R).
\end{equation} 

The next result is implicitly known from \cite{LyuMa}, and we give a direct proof for the sake of completeness.

\begin{lemma}\label{lem-Dr-map-ad-inv}
The Drinfeld ismorphism is an $H$-intertwining map $\jmath_l: H_\ad^* \lra H_\ad$.
\end{lemma}
\begin{proof}
Below we write
\[
R_{21}R=\sum R^\prime_2 R_1 \otimes R^\prime_1 R_2
\]
to differentiate the two different copies of $R$. 

For an $h\in H$,  we compute
\begin{align*}
\jmath_l(\ad^*h(f)) & = \jmath_l (f(\ad S^{-1}(h)(\mbox{-})))=\sum f( S^{-1}(h)_1 S^{-1}(R_2^\prime R_1) S(S^{-1}(h)_2) ) R_1^\prime R_2 \\
& = \sum f(S^{-1}(h_2)  S^{-1}(R_2^\prime R_1)h_1) R_1^\prime R_2 = \sum f(S^{-1}(R_2^\prime R_1 h_2)h_1)R_1^\prime R_2 h_3 S(h_4)\\
& = \sum f(S^{-1}( R_2^\prime h_3 R_1 )h_1) R_1^\prime h_2 R_2 S(h_4)=\sum f(S^{-1}(h_2R_2^\prime R_1)h_1) h_3 R_1^\prime R_2 S(h_4) \\
& = \sum f(S^{-1}(R_2^\prime R_1) S^{-1}(h_2)h_1) h_3R_1^\prime R_2 S(h_4) =\sum f(S^{-1}(R_2^\prime R_1)) h_1 R_1^\prime R_2 S(h_2)\\
&=f(S^{-1}(R_2^\prime R_1)) \ad h (R_1^\prime R_2) = \ad h (\jmath_l(f)).
\end{align*}
Here we have repeated used the first condition of the quasi-triangular structure that 
\[
\sum R_1 h_1 \otimes R_2 h_2 = \sum h_2 R_1 \otimes h_1 R_2.
\]
The Lemma follows.
\end{proof}

\begin{definition}\label{def-Fourier}
Given a finite-dimensional factorizable Hopf algebra $H$, the  \emph{Fourier transform on $H$} is the composition
  $$\mathcal{F}=\jmath_l\circ \psi_l:H\lra H$$
  of the Drinfeld and Radford maps. Likewise, the \emph{Fourier transform on $H^*$} is the composition
    $$\mathcal{F}^*=\psi_l\circ \jmath_l:H^*\lra H^*.$$
\end{definition} 

\begin{corollary}\label{cor-F-perserves-Higman}
The Fourier transforms are isomorphisms of $H$-modules under the adjoint actions:
\[
\mathcal{F}: H_\ad \stackrel{\cong}{\lra} H_\ad,\quad \quad \mathcal{F}^*:H_\ad^* \stackrel{\cong}{\lra} H_\ad^*.
\]
Consequently, the Fourier transform restricts to automorphisms 
$$\mathcal{F}:\zl(H)\lra \zl(H), \quad \quad \mathcal{F}^*:\cl_l(H)\lra \cl_l(H),$$ 
and the Higman ideal $\zl_{\Hig}(H)$ is invariant under the Fourier transform: 
\[  \mathcal{F}( \zl_{\Hig}(H)) = \zl_{\Hig}(H) . \] 
\end{corollary}
\begin{proof}
The first statement follows from Lemma \ref{lem-Rad-ad-inv} and Lemma \ref{lem-Dr-map-ad-inv}. Using the first statement, the Fourier transforms $\mathcal{F}$ and $\mathcal{F}^*$ preserve $\ad H$-invariant subspaces of $H_\ad$ and $H_\ad^*$, which are identified with $\zl(H)$ and $\cl_l(H)$ respectively via Lemma \ref{ad-H-on-itself} and Lemma \ref{lem-Rad-ad-inv}. Finally, the $\ad$-invariance of $\mathcal{F}$ gives us
\[
\mathcal{F}(\ad \Lambda(H))= \ad \Lambda (\mathcal{F}(H))=\ad \Lambda (H).
\]
The result follows.
\end{proof}

\begin{definition}\label{def-HC-center}
Let $H$ be a quasi-triangular Hopf algebra. The \emph{Harish-Chandra center} of $H$ is the subspace of $\zl(H)$
\[
\zl_\HC( H ):= \jmath_l(\rl_l(H)),
\] 
\end{definition}

We document the following properties of $\zl_{\Hig}(H)$ and $\zl_{\HC}(H)$, which are known in \cite{CoWe}. We give the proofs for completeness while weakening some assumptions.

\begin{proposition}\label{prop-Higman-in-HC}
Let $H$ be a factorizable Hopf algebra. 
\begin{enumerate}
\item[(i)] The
Higman ideal of $H$  is spanned by the images under the Drinfeld map of the characters of the projective modules: 
\[ 
\zl_\Hig(H) = \jmath_l (\pl_l(H)). 
\] 
In particular, this implies that $\zl_{\Hig}(H)$ is an ideal in $\zl(H)$ of dimension equal to the rank of the Cartan matrix of $H$.  
\item[(ii)] If  $\mathrm{dim}(H)$ is invertible in $\Bbbk$, then the Higman ideal always contains a  nonzero idempotent.
\item[(iii)] The Higman ideal $\zl_\Hig(H)$ is contained in the intersection of the Harish-Chandra center and its Fourier transform: 
$$\zl_\Hig(H)\subset \zl_{\HC}(H)\cap \mathcal{F}(\zl_{\HC}(H)).$$
\item[(iv)] The Fourier transform of $\zl_{\HC}(H)$ is also equal to $\mathcal{F}(\zl_{\HC}(H))=\psi_l^{-1}(\rl_l(H))$.
\end{enumerate}
\end{proposition}
\begin{proof}
$(i)$. Since the Higman ideal is invariant under the Fourier transform (Corollary \ref{cor-F-perserves-Higman}), we have
\[
\zl_{\Hig}(H)=\mathcal{F}(\zl_{\Hig}(H)) = \mathcal{F}(\psi_{l}^{-1}(\pl_l(H)))=\jmath_l(\pl_l(H)) .
\]
As the Drinfeld map $\jmath_l$, when restricted to the center, is an isomorphism of commutative algebras, and $\zl_{\Hig}(H)$ is an ideal in $\zl(H)$, we deduce that $\pl_l(H)$ is an ideal in $\cl_l(H)$ of dimension $\mathrm{rank}(C_H)$ (see Example \ref{eg-shifted-traces} (2)). The dimension count then follows.

$(ii)$. Since $\jmath_l$ is an isomorphism of commutative algebras, it suffices to show that $\pl_l(H)$ always has an idempotent. This follows from the fact that tensor product module decomposition $H\otimes H\cong H^{\mathrm{dim}(H)}$. Thus, if $\mathrm{char}(\Bbbk)\nmid \mathrm{dim}(H)$, then $\frac{1}{\mathrm{dim}(H)}\chi_H(u\mbox{-})$ is an idempotent in $\pl_l(H)$. 

$(iii)$. Next, as $\pl_l(H)\subset \rl_l(H)$ by definition of (shifted) characters (see equation \eqref{eqn-characters}), we have the inclusion
\[ \zl_{\Hig}(H) =\jmath_l(\pl_l(H))  \subset \jmath_l(\rl_l(H)) = \zl_{\HC}(H). \] 
Again, using the invariance of $\zl_{\Hig}$ under $\mathcal{F}$ gives us
\[
\zl_{\Hig}(H)=\mathcal{F}(\zl_{\Hig}(H))\subset \mathcal{F}(\zl_{\HC}(H)).
\]
The result now follows.

$(iv)$.  For the last result, we will use that \cite[Theorem 5.1]{La}
\[
 \mathcal{F}^2\big|_{\zl(H)}=S|_{\zl(H)}.
\]
Then we have
\[
\mathcal{F}^2(\zl_{\HC}(H))=S(\zl_{\HC}(H)).
\]
As $S$ sends the character of a simple $H$-module to the character of its dual module, it follows that $S$ preserves $\zl_\HC$, and
\[
\jmath_l\psi_l(\mathcal{F}(\zl_{\HC}(H)))=\zl_{\HC}(H).
\]
Taking $\jmath_l^{-1}$ on both sides gives the desired equality.
\end{proof}

\begin{corollary} \label{cor-annhilator-of-radical}
For a factorizable Hopf algebra $H$,
 the subspace $\mathcal{F}(\zl_{\HC}(H))$ is an ideal in $\zl(H)$ that is annihilated by the radical of $\zl(H)$. 
\end{corollary}
\begin{proof}
By part $(iv)$ of Proposition \ref{prop-Higman-in-HC}, we may identify $\mathcal{F}(\zl_{\HC}(H))=\psi_l^{-1}(\rl_l(H))$. We will also need the fact that (Theorem \ref{thmRadfordiso}) $\psi_l : \zl(H) \to \cl_l(H)$ is an isomorphism of $\zl(H)$-modules. 

Let  $z \in \zl(H)$, and consider 
the functional $\chi_L(K_{2\rho}(\mbox{-})z)$, where $L$ is a simple $H$-module. Since central elements act by scalars in simple modules, the resulting functional is a scalar $\alpha_z$ multiple of $\chi_L(K_{2\rho}(\mbox{-}))$, and $\alpha_z$ must be zero if $z$ is nilpotent. Therefore 
\[ 
 z\psi_l^{-1}(\chi_L) = \alpha_z \psi^{-1}(\chi_L)  
\]
with $\alpha_z =0$ if $ z $ is nilpotent. This shows that $\psi_l^{-1}({\rl}_l ) = \mathcal{F}(\zl_\HC(H))$ is an ideal in $\zl(H)$ that is annihilated by the nilradical of $\zl(H)$.  
\end{proof}

\begin{remark}
We would like to correct the mistake in \cite{La} and point out that the subspace annihilating the radical of $\zl(H)$ does not in general coincide with $\mathcal{F}(\zl_{\HC}(H))$, but contains it as a subspace. 
\end{remark}

\paragraph{Modular group action on the adjoint representation.}
We recall another important notion from \cite{DrACHA}.
\begin{definition}\label{def-ribbon-Hopf}
A factorizable Hopf algebra $H$ is called \emph{ribbon} if there is a central element $v\in H$ such that
\[
S(v)=v \quad \quad v^2=uS(u), \quad \quad \Delta(v)(R_{21}R)=v\otimes v.
\]
\end{definition}

 When $H$ is a factorizable ribbon Hopf algebra, Lyubashenko and Majid \cite{LyuMa} have introduced an extended  modular group action on $H_\ad$. 

Define 
\begin{equation}
\mathcal{L}:H  \lra H, \quad h\mapsto v h,
\end{equation}
where $v$ is the ribbon element of $H$.

\begin{theorem}\label{thm-modular-group-action}
Let $H$ be a factorizable ribbon Hopf algebra. Then the maps
$
\mathcal{F},~\mathcal{L}
$
define an extended modular group action on $H_\ad$. Furthermore, this action descends to a projective $SL_2(\Z)$ action  on 
$\zl(H)$ and on $\zl_{\Hig}(H)$\footnote{This result is first stated in \cite{LMSS}, and the authors would like to thank C.~Schweigert for pointing out the references to us. }.  
\end{theorem}
\begin{proof}
 Corollary \ref{cor-F-perserves-Higman} shows that $\mathcal{F}$ preserves the $H$-module structure on $H_\ad$. Next, note that, since $v\in H$ is central, left multiplication by $v$ (or any central element of $H$) defines an endomorphism of $H_\ad$. The fact that $\mathcal{F}$ and $\mathcal{L}$ satisfy the extended $SL_2(\Z)$ relations follows from \cite[Theorem 1.1]{LyuMa}, and in a slightly modified, but equivalent, form (see \cite[Theorem 5.1]{La}). Furthermore, this action preserves the center $\zl(H)$, since $\mathcal{F}$ preserves the subspace of $\ad$-invariants in $H_\ad$ and $v \in \zl(H)$ is central. Finally, $\mathcal{F}$ preserves  $\zl_{\rm Hig}(H)$ again by Corollary \ref{cor-F-perserves-Higman} and $v$ preserves it since $\zl_{\rm Hig}(H) \subset \zl(H)$ is an ideal. Therefore the same action preserves $\zl_{\rm Hig}(H)$. 
 
The extended modular group action descends to a projective $SL_2(\Z)$ action on the center as shown in \cite{LyuMa} (see also \cite[Theorem 5.1]{La}). This is, essentially, because $S^2$ acts on $H$ as conjugation by the invertible element $u\in H$. 
\end{proof}

\section{Derived center of the small quantum group as a $\g$-module}
\label{sec-q-group}
\paragraph{Quantum groups at roots of unity.}
For the rest of the paper, we will take $\Bbbk=\C$ to be the field of complex numbers, and fix $q\in \C$ an primitive $l$th root of unity.

Given a complex semisimple Lie algebra $\g$, denote by $I=\{\alpha_i\}$ its simple root system. To define the quantum groups at a root of unity, we will always make the assumption that 
\begin{itemize}
\item $q$ is of odd order $l$ which is greater than or equal to the Coxeter number of $\g$, and coprime to the determinant of the Cartan matrix of $\g$.
\end{itemize}

The algebras $\Ul_q(\g)$, $U(\g)$ and $\ul_q(\g)$ carry a Hopf algebra structure satisfying 
\[ 
\Delta(E_{i}) = E_{i} \otimes 1 + K_i \otimes E_{i},  \quad 
\Delta(F_{i}) = 1 \otimes F_{i} + F_{i} \otimes K_i^{-1}, \quad  \Delta(K_i) = K_i \otimes K_i 
\] 
 and 
\[ S(E_{i}) = - K_i^{-1} E_{i},  \quad  S(F_{i}) = - F_{i} K_i , \quad   S(K_i) = K_i^{-1} \] 
 for all $i = 1 , \dots, r$.
 
When no confusion can be caused, we will usually abbreviate $\Ul=\Ul_q(\g)$, $U=U(\g)$ and $\ul=\ul_q(\g)$ in what follows.

\paragraph{A $U$-action on $\zl(\ul)$.} 
Now we will define an action of the universal enveloping algebra $U(\g)$ on the center of the corresponding small quantum group $\zl(\ul)$ arising from the Hopf-adjoint action of $\Ul$ on itself. 

\begin{lemma}  \label{ad_on_ul} 
The Hopf-adjoint action of $\Ul$ on itself preserves the small quantum group $\ul$. In other words, we have $\ul_\ad \subset \Ul_{\mathrm{ad}}$ as a $\Ul$-submodule. 
\end{lemma} 

\begin{proof} We only need to check that the action of the $l$-th divided powers $\{E_i^{(l)} , F_i^{(l)} \}_{i=1}^r$ preserves $u$. 
 We use the formulas derived in the series of papers \cite{LusModularRepQGp}, \cite{LusfdHopf}, \cite{Lusrootsof1}. 
\[ \Delta (E_i^{(l)})  = \sum_{0 \leq k \leq l} q^{d_i k (l-k)} E_i^{(l-k)} K_i^k \otimes E_i^{(k)} .\] 
Let $a \in \ul$. Then we have  
\begin{align*} 
{\mathrm{ad}} E_i^{(l)} ( a)& =  \sum E_{i, 1}^{l} \; a\; S(E_{i, 2}^{(l)})  =   \sum_{0 \leq k \leq l} q^{d_i k (l-k)} E_i^{(l-k)} K_i^k \; a \;  S(E_i^{(k)}) \\
& = \sum_{0 \leq k \leq l} q^{d_i k (l-k)} E_i^{(l-k)} K_i^k \; a \; (-1)^k  q^{d_i k(k-1)} K_i^{-k} E_{i}^{(k)}  \\
& = E_i^{(l)} a - a E_i^{(l)} + \sum_{0 < k < l} q^{d_i k ( l-1)} (-1)^k E_i^{(l-k)} K_i^k \; a \;  K_i^{-k} E_{i}^{(k)} .
\end{align*}
The last term belongs to $\ul$. The commutation relations in $\Ul$ given in \cite{LusfdHopf, Lusrootsof1} ensure that 
for any $x \in \ul$, such that $x$ is one of the generators $\{ E_j, F_j, K_j^{\pm1} \}$ we have 
\[ E_i^{(l)} x = x  E_i^{(l)} +  {\mathrm{terms}} \;\;{\mathrm{in }}  \;\; \ul .\] 
For example, we have in the case $\langle \alpha_i, \alpha_j \rangle =-1$ in the simply laced case: 
\[ E_i^{(l)} E_j = E_j E_i^{(l)} + E_i E_j E_i^{(l-1)} , \quad \quad  E_i^{(l)} F_i = F_i E_i^{(l)} + \frac{K_i q^{d_i} - K_i^{-1} q^{-d_i}}{q^{d_i} - q^{-d_i}} .\]  
Therefore for any $a \in \ul$ we have 
\[ E_i^{(l)} a = a E_i^{(l)} + {\mathrm{terms}} \;\;{\mathrm{in }}  \;\; \ul .\] 
The computation for $ {\mathrm{ad}} F_i^{(l)} (a) $ is similar. 
\end{proof} 

Lemmas  \ref{ad-H-on-itself} and \ref{ad_on_ul} show that $\ul$ is a $\Ul$-module algebra with respect to the adjoint action. Furthermore, it is a normal Hopf subalgebra in the sense of Definition \ref{def-hopf-quotient}. The Hopf quotient algebra can be identified, via the quantum Frobenius map \eqref{eqn_quantumFrob}, with the classical universal enveloping algebra of $\g$: 
\begin{equation}
\Ul_q(\g)\qq\ul_q(\g)=U(\g).
\end{equation}

\begin{lemma}  \label{ad_on_center}  
\begin{enumerate} 
\item[(i)] The Hopf-adjoint action of the $l$-th divided powers preserves the center of the small quantum group $\zl(\ul)$. 
\item[(ii)] For any $z \in \zl(\ul)$  and for all $i= 1, \ldots r$ we have  
\[ {\mathrm{ ad}} E_i^{(l)} ( z) =  E_i^{(l)} z - z E_i^{(l)} , \quad \quad  {\mathrm{ ad}}  F_i^{(l)} ( z) = F_i^{(l)} z - z F_i^{(l)}  .  \] 
\end{enumerate}  
\end{lemma} 

\begin{proof}  
The first statement follows from Lemmas \ref{ad_on_ul}, \ref{ad-H-on-itself}  and  \ref{H-mod-alg}: since $\ul$ is a $U$-module algebra with respect to the adjoint action, this action preserves the center of $\ul$.  In particular the action of the $l$-th divided powers preserves the center of $\ul$. 

Now let $z \in \zl(\ul)$. Then from the proof of Lemma \ref{ad_on_ul} we have  
\begin{align*}  
{\mathrm{ad}} E_i^{(l)} (z) & =  E_i^{(l)} z - z E_i^{(l)} + \sum_{0 < k < l} q^{d_i k ( l-1)} (-1)^k E_i^{(l-k)} K_i^k  z  K_i^{-k} E_{i}^{(k)}  \\
& = E_i^{(l)} z - z E_i^{(l)} +  \sum_{0 < k < l} q^{d_i k ( l-1)} (-1)^k  \frac{[l]_{d_i}!}{[k]_{d_i}![l-k]_{d_i}!}  E_i^{(l)}  \; z  =  E_i^{(l)} z - z E_i^{(l)} .
\end{align*}
 The computation for $ {\mathrm{ad}}  F_i^{(l)} ( a)$ is similar. 

 \end{proof} 

\begin{theorem} \label{Action-div-powers}
The Hopf-adjoint action of $\Ul_q(\g)$ on $\ul_q(\g)$ induces a $U(\g)$-module structure on the center $\zl$ of the small quantum group $\ul$ via the quantum Frobenius homomorphism. Furthermore, the action is given by taking commutators with the generators $E_i^{(l)}$, $F_i^{(l)}$, $i = 1, \ldots r$.
\end{theorem} 

\begin{proof} 
We use the quantum Frobenius homomorphism of equation \eqref{eqn_quantumFrob}, identifying $U(\g)$ as the quotient $\Ul\qq\ul$ under the quantum Frobenius map $\phi$ (for the detailed treatment see \cite{Lusrootsof1}, section 8). By Lemma \ref{ad_on_center} the generators $E_i^{(l)}, F_i^{(l)}$, $i = 1, \ldots r$ preserve the center $\zl(\ul)$ under the adjoint action. On the other hand, the adjoint action of the small  quantum group $\ul$ is trivial on its center. Therefore, the action factors through the Frobenius homomorphism and gives rise to the action of a (completion of) the universal enveloping algebra $U(\g)$. 
\end{proof} 

According to the linkage principle \cite{APW, APW2}, the category of representations of $\ul$ decomposes into a direct sum of blocks 
\begin{equation} \label{eqn-block-decomp}
{\rm Rep}(\ul_q(\g)) \cong \bigoplus_{\lambda \in P/(W \ltimes lP) } {\rm Rep}(\ul_\lambda) 
\end{equation}
The blocks are parametrized by the orbits  of the extended affine Weyl $W \ltimes lP$ group on the weight lattice $P$.  The block acting nontrivially on the trivial representation of $\ul$ is called the \emph{principal block} of the category, and denoted $\ul_0$ in what follows. The Jantzen translation principle tells us that, if the stablizer subgroup of a weight $\lambda\in P$ in $W$ is trivial, then $ {\rm Rep}(\ul_\lambda) $ is equivalent to $ {\rm Rep}(\ul_0) $.

\begin{corollary}\label{cor-Ug-preserve-block}
The Hopf-adjoint action of $\Ul_q(\g)$ on $\ul_q(\g)$ preserves the block decomposition of $\ul_q(\g)$.
\end{corollary}
\begin{proof}
By \cite[Proposition 4.2]{La}, the central idempotents of $\ul$ arise from (Lagrangian interpolations of) Harish-Chandra elements of $\Ul$ restricted to the small quantum group. Such central elements are then $\Ul$-adjoint invariant.
\end{proof}

\paragraph{Hopf cohomology and $\g$-action on the derived center of $\ul$.}
Now we will define the action of $U=U(\g)$ on the total Hochschild cohomology of $\ul$. We will show that, restricted to the 
center $\mHH^0(\ul)$, this action coincides with the $U$-action defined in the previous section.

We recall a result of Ginzburg-Kumar \cite{GinKu} which interprets the Hochschild cohomology groups, or {\em the derived center} of a Hopf algebra $H$, as its Hopf cohomology with coefficients in the adjoint representation $H_\ad$. Let us denote by
$H^{\mathrm{e}}:=H\otimes H^{\mathrm{op}}$ the \emph{enveloping algebra} of $H$, and the algebra embedding
\begin{equation}
\delta: H\lra H^{\mathrm{e}},\quad h\mapsto \sum h_1\otimes S(h_2).
\end{equation}
Clearly, restriction of the natural bimodule structure of $H$ along $\delta$ gives rise to $H_\ad$. Furthermore, it is an easy exercise to show that $H^{\rm e}$ is a free module over $H$ along $\delta$, and
\begin{equation}\label{eqn-Had-from-bimod}
H\cong H^{\rm e}\otimes_{\delta, H} \Bbbk \cong \Bbbk\otimes_{H, \delta} H^{\rm e}.
\end{equation}

\begin{theorem}[Ginzburg-Kumar]   \label{HH-Hopf}
Given a Hopf algebra $H$, there is an algebra isomorphism
\[
\mHH^\bullet(H)\cong {\rm Ext}^\bullet_{H^{\rm e}} (H, H)\cong {\rm Ext}^\bullet_{H} (\Bbbk, H_\ad) \cong \mH^\bullet(H,H_\ad). 
\]
\end{theorem}
\begin{proof}[Sketch of proof ](For details, see \cite[Section 5.6]{GinKu}) The two outer isomorphisms are just definitions of the cohomology groups. It suffices to show the middle one. To do this, it suffices to show that $\Ext^\bullet_{H^{\rm e}}(H,\mbox{-})$ and $\Ext^\bullet_H(\Bbbk,\mathrm{Res}^{H^{\rm e}}_H(\mbox{-}))$ are isomorphic as derived functors on the category of $H^{\rm e}$-modules. Now using derived induction-restriction adjunction and equation \eqref{eqn-Had-from-bimod}, we have:
\[
\Ext^\bullet_{H^{\rm e}}(H,M)\cong \Ext^\bullet_{H^{\rm e}}(\mathrm{Ind}^{H^{\rm e}}_{H} \Bbbk, M)\cong \Ext^\bullet_{H^{\rm e}}( \Bbbk, \mathrm{Res}^{H^{\rm e}}_{H} M).
\]
Applying the isomorphism to $M=H$ gives us the desired isomorphism. The algebra structure $\mH^\bullet(H,H_\ad)$ arises from the multiplication of $m:H_\ad \otimes H_\ad \lra H_\ad$, and it is an easy exercise to show that it is compatible with the isomorphism.
\end{proof}

Although it is not explicitly stated in \cite{GinKu}, the following result is obtained similarly. 

\begin{corollary}\label{cor-Hochschild-Homology}
Given a Hopf algebra $H$, there is an isomorphism
\[
\mHH_\bullet(H)=\mathrm{Tor}_\bullet^{H^{\rm e}}(H,H)\cong \mathrm{Tor}^{H}_\bullet(\Bbbk, H_\ad).
\]
\end{corollary}
\begin{proof}
Similar as in the proof of the previous result, we have, using equation \eqref{eqn-Had-from-bimod}, 
\[
\mathrm{Tor}_\bullet^{H^{\rm e}}(H,H)\cong (\Bbbk \otimes_{H,\delta} H^{\rm e})\otimes^{\mathbf{L}}_{H^{\rm e}} H\cong \Bbbk \otimes^{\mathbf{L}}_{H,\delta} H\cong\mathrm{Tor}^H_\bullet(\Bbbk, H_\ad).
\]
The result follows.
\end{proof}

For finite dimensional Hopf algebras, there is a well-known duality between Hochschild homology and cohomology groups.

\begin{lemma}
Let $H$ be a finite dimensional Hopf algebra. Then there is an isomorphism of graded vector spaces
\[
\mHH^{\bullet}(H)\cong (\mHH_{-\bullet} (H))^*.
\]
\end{lemma}
\begin{proof}
We have, via the derived tensor-hom adjunction,
\[
(\Tor_\bullet^H(\Bbbk, H_\ad))^*\cong \Hom_\Bbbk(\Bbbk \otimes_H^{\mathbf{L}} H_\ad,\Bbbk)\cong\mathbf{R}\Hom_H^\bullet(\Bbbk, \Hom_\Bbbk(H_\ad,\Bbbk))\cong \mathbf{R}\Hom_H(\Bbbk, H_\ad).
\]
The last isomorphism follows from the $\ad$-invariance of the Radford isomorphism \ref{lem-Rad-ad-inv}. The lemma follows.
\end{proof}

For this reason, we will mostly focus on the Hochschild cohomology groups of small quantum groups from now on. 

When $H=\ul_q(\g)$, equipped with the left adjoint action by the big quantum group $\Ul_q(\g)$, Theorem \ref{HH-Hopf} immediately implies the following result due to Ginzburg-Kumar, whose proof we recall for completeness.

\begin{corollary}   \label{GK-action} 
\begin{enumerate} 
\item[(i)] There is a natural action of the universal enveloping algebra $U(\g)$ on the Hochschild cohomology $\mHH^\bullet(\ul)$ and homology $\mHH_\bullet(\ul)$ of the small quantum group. 
\item[(ii)] Restricted to the center $\zl(\ul)=\mHH^0(\ul)$, this action coincides with the action given in Theorem 
\ref{Action-div-powers}.
\end{enumerate}  
\end{corollary} 

\begin{proof} 
$(i)$. Suppose $H\subset A$ is a normal Hopf subalgebra, such that $A$ is flat as an $H$-module. Then $\mH^\bullet(H,H_\ad)$ carries a natural $A\qq H$ action, defined via the isomorphism:
\[
\mH^\bullet (H, H_\ad)\cong \mathrm{Ext}_{H}^\bullet(\Bbbk, H_\ad)\cong \mathrm{Ext}^\bullet_A (A\otimes_H \Bbbk, H_\ad)\cong \mathrm{Ext}^\bullet_A(A \qq H,H_\ad).
\]
Here the middle isomorphism is the usual (derived) induction-restriction adjunction, where the tensor in $\mathrm{Ext}^\bullet_A (A\otimes_H \Bbbk, H_\ad)$ does not need to be derived because of the flatness assumption on $A$.

The isomorphism, applied to $H=\ul$ and $A=\Ul$, implies that $U(\g)\cong \Ul \qq \ul$ acts on the Hochschild cohomology $\mHH^\bullet(\ul)$. 

  $(ii)$. The action restricted to the center $\mHH^0(\ul)$ is given by the adjoint action of the $l$-th divided powers 
(via the quantum Frobenius map $\phi$ (equation \eqref{eqn_quantumFrob})).  
\end{proof} 

\begin{corollary}   \label{H_in_HH}  
Let $\Nc \subset \g$ denote the nilpotent cone of $\g$. Then
we have the inclusion $  \C[\Nc] \subset  \mHH^\bullet(\ul) $ and $\C[\Nc] \subset  \mHH_\bullet(\ul)$.  In particular, $\C[\Nc]$ is a $U(\g)$-summand of $\mHH^\bullet(\ul)$ and $ \mHH_\bullet(\ul)$ with respect to the standard (co)adjoint action. 
\end{corollary} 

\begin{proof} 
Let $L(0)$ denote the trivial $\ul$-module. We have, by Theorem \ref{HH-Hopf}, 
\[ \mHH^\bullet(\ul) \cong  \mH^\bullet(\ul, \ul_{\rm ad}) \cong {\rm Ext}_\ul^\bullet(L(0), \ul_{\rm ad}) \supset {\rm Ext}_\ul^\bullet(L(0), L(0) ) \cong \mH^\bullet(\ul), \]
where $\mH^\bullet(\ul) $ stands for the usual cohomology of the Hopf algebra $\ul$ (with coefficients in the trivial representation).
 
The main result in \cite{GinKu} states that the odd cohomologies $\mH^{\rm odd}(\ul)$ vanish, and there is a graded algebra isomorphism  $\mH^{2\bullet}(\ul)  \cong \C^\bullet[\Nc]$ for the even cohomologies of $\ul$. Moreover, this isomorphism intertwines the $\g$-actions on both sides. The $\g$-action on  $\mH^\bullet(\ul)$ is induced from the Frobenius pullback of $\g$ just as in Corollary \ref{GK-action}. The $\g$-action on $\C[\Nc]$ is induced from  the standard (co)adjoint action on 
$\Nc$. 
\end{proof}

\paragraph{The action of $\g$ in the geometric realization of the derived center.}
Recall that the finite dimensional Hopf algebra $\ul$ decomposes as a direct sum of blocks: two-sided ideals parametrized by the orbits of the extended affine Weyl group $W \ltimes lP$ in the weight lattice $P$. Denote by $\ul_\lambda \subset \ul$ the unique block corresponding to the orbit of the weight $\lambda$.

In this section we will use the geometric realization of the total Hochschild cohomology of the block $\ul_\lambda$  
to construct a natural $\g$-action on $\mHH^\bullet(\ul_\lambda)$. Restricted to the center of the block $\zl(\ul_\lambda)$ this action coincides with the action defined in the previous two sections. 

Recall the geometric construction for  $\mHH^\bullet(\ul_\lambda)$ described in \cite{ABG} \cite{BeLa}. Let $G$ be the reductive algebraic 
group over $\C$ with the Lie algebra $\g$, $P_\lambda$ its fixed parabolic subgroup whose Weyl group stablizes $\lambda$, $X=G/P_\lambda$ the (partial) flag variety classifying subgroups conjugate to $P_\lambda$. Let 
$\Nt_\lambda \cong T^* X \cong G \times_{P_\lambda} \n$ denote the Springer resolution, where $\n$ stands for the nilpotent radical of the Lie algebra of $P_\lambda$. 
Elements in $\Nt_\lambda$ are given by pairs $(g,x)$, where $g\in G$ and $x\in \n$, subject to the identification $(g,x)=(gb^{-1}, \mathrm{Ad}_b(x)))$. 
%Let $\op{pr}: T^*X \lra X$ be the canonical projection map that sends the equivalence class of the pair $(g,x)$ to the coset $gB$. It is evidently $G$-equivariant. 
Let the group $\mathbb{C}^*$ act on $G/P_\lambda$ trivially, and define its action on $\Nt_\lambda$ by rescaling the fibers of $\op{pr}:\Nt_\lambda \lra G/P_\lambda$, which are all isomorphic to the vector space $\n$, via the character $z\mapsto z^{-2}$. This action commutes with the action of $G$ on $\Nt_\lambda$ and $G/P_\lambda$. 
%It is easy to check that, with respect to these actions, $\op{pr}$ is in fact $G\times \mathbb{C}^*$-equivariant. 

\begin{theorem}     \label{BL} 
Let $\ul_\lambda \subset \ul$ be the block of $\ul$ that corresponds to the weight $\lambda$. Then we have 
\[
\mHH^\bullet(\ul_\lambda)\cong \mHH^\bullet_{\C^*}(\Nt_\lambda)\cong \bigoplus_{i+j+k=\bullet} \mH^i(\Nt_\lambda,\wedge^j T\Nt_\lambda)^k
\]
where $k$ is the grading induced by the $\C^*$-action. 
\end{theorem} 
\begin{proof}
See \cite{ABG, BeLa}. The singular weight case, stated in \cite{LQ2}, follows similarly as in \cite{BeLa} using the singular localization theorem of Backelin-Kremnizer \cite{BaKr, BaKrSing}.
\end{proof}

\begin{theorem}  \label{thm-Schouten-action} 
There is a natural action of the Lie algebra $\g$ on the total Hochschild cohomology of the principal block of $\ul$. Restricted to the center of the principal block, this action coincides with the action given in Theorem \ref{Action-div-powers}  and Corollary \ref{GK-action}. 
\end{theorem} 

\begin{proof} 
We have shown in \cite{LQ1,LQ2}, that 
\[
\mHH^1(\ul_\lambda)\supset \mH^0(\Nt_\lambda,T\Nt_\lambda)^0 \cong \g\oplus \C.
\]
Global sections of $T\Nt_\lambda$ consists of vector fields on $\Nt_\lambda$. Vector fields act, via the Schouten bracket, on $\wedge^\bullet T\Nt_\lambda$. This in turn induces an action of $\g\oplus \C$ on the Hochschild cohomology groups (The $\C$ part comes from the Euler vector field generated by the $\C^*$-action along the fibers, and counts the $k$-degree of Hochschild cohomology group elements). The Schouten bracket is defined by extending the natural commutator of germs of vector fields on $\Nt_\lambda$ to germs of poly-vector fields. By the main Theorem of \cite{CaVdB}, this Schouten bracket defines a Gerstenharber structure on the Hochschild cohomology of $\Nt_\lambda$.

On the other hand, the $\g$ action on $\Nt_\lambda$ also arises as the infinitesimal $G$-action on the variety $\Nt_\lambda$, and the Springer resolution $\pi:\Nt_\lambda\lra \Nc_\lambda$, $(g,x)\mapsto \mathrm{Ad}_b(x)$ is $G$-equivariant. Moreover, $\mHH^\bullet_{\C^*}(\Nt_\lambda)$ is a $\g$-equivariant algebra over the function algebra $\C[\Nt_\lambda]$. Hence, the Gerstenhaber action of $\g$ on $\mHH^\bullet(\ul_\lambda)\cong \mHH^\bullet_{\C^*}(\Nt_\lambda)$ is compatible with the $\g$ action on $\C[\Nt_\lambda]= \C[\Nc_\lambda]$. As $\iota:\Nc_\lambda\subset \Nc$ is a closed $G$-orbit, the $\g$ action on $\C[\Nc_\lambda]$ agrees with the $\g$-action on $\C[\Nc]$ under $\iota^*$. By Ginzburg-Kumar's Theorem \ref{GK-action}, we identify this $\g$-action with the (derived) Hopf adjoint action by the $l$-th divided powers of generators of the big quantum group $\Ul_q(\g)$. 
\end{proof}

\begin{remark} 
By the result of \cite{CaVdB} the action of $\mH^0(\Nt_\lambda, T\Nt_\lambda)^0\subset \mHH^1(\ul_\lambda)$ on $\mHH^\bullet(\ul_\lambda)$ by the Gerstenhaber bracket agrees, up to a twist by the square root of the Todd class, with the action of the vector field $\g \oplus \C$ on 
$\mHH^\bullet_{\C^*}(\Nt_\lambda)$  by the Schouten bracket. 
\end{remark}

\paragraph{The action of $\g$ on the center in type $A$.} 
In \cite{LQ1,LQ2}, we have computed that,  for all blocks of $\ul_q(\s_n)$ when $n=2,3,4$, the natural $\g$-modules occurring $\zl(\s_n)$ from Theorem \ref{thm-Schouten-action} only consist of trivial representations. Since, by Theorem \ref{Action-div-powers}, the nilpotent Chevalley generators of $\g$ acts on the center by taking commutators with $E_i^{(l)}$ and $F_i^{(l)}$, the triviality of the $\g$-action means that the central elements in $\zl(\s_n)$ commutes with $E_i^{(l)}$ and $F_i^{(l)}$. We conjecture that this is not a coincidence in type A.

\begin{conjecture}  \label{triv-g-action}
At a root of unity, central elements of small quantum groups $\ul_q(\s_n)$ arise from restriction of central elements in the big quantum group $\Ul_q(\s_n)$.
 \end{conjecture}
 
The conjecture fails outside of type $A$, as shown in \cite[Appendix 2]{LQ2}. There are already nontrivial $\g$-modules appearing in the center of $\ul_q(\g)$ in type $B_2$.

\section{Higman ideal in the center of the small quantum group}
\label{sec-Hig-ideal}
In the section we derive further results on modular group action and the Higman ideal in case when $H$ is a factorizable ribbon Hopf algebra, or more specifically, the small quantum group. 

\paragraph{Modular group action on a factorizable ribbon Hopf algebra.}  
If $H$ is a factorizable ribbon Hopf algebra, the extended modular group action on $H_\ad$ (Theorem \ref{thm-modular-group-action}) descends to an action on the Hochschild cohomology. This result has previously been noted in \cite{LMSS, ScWo}. In our approach it is also a direct consequence of Ginzburg-Kumar's Theorem \ref{HH-Hopf}.

\begin{corollary}
Let $H$ be a factorizable ribbon Hopf algebra. Then there is a projective $SL_2(\Z)$-action on $\mHH^\bullet(H)$. 
\end{corollary}
\begin{proof}
This follows by combining Theorem \ref{HH-Hopf} and Theorem \ref{thm-modular-group-action}. 
\end{proof}

The small quantum group $\ul$ is factorizable ribbon (see, for example, \cite[Corollary A.3.3, Theorem A.4.1. and Appendix B]{Lyu}). When $H=\ul$ and we restrict the projective action of the modular group to $\mHH^0(\ul)=\zl$, the action agrees with the one studied in \cite{La}. This action, unlike in Corollary \ref{cor-Ug-preserve-block}, preserves neither the algebra structure of $\zl$ nor the block decomposition of $\zl$. This is because the generating maps $\mathcal{F},~\mathcal{L}:H_\ad \lra H_\ad$ and do not preserve the multiplicative structure of $H_\ad$.

\begin{corollary}  \label{hig_ideal_HH}
The Higman ideal constitutes a homogeneous ideal in the Hochschild cohomology ring $\mHH^\bullet(H)$. 
Furthermore, it is a direct summand, as a projective $SL_2(\Z)$-module, in the Hochschild cohomology ring.
\end{corollary}
\begin{proof}
Under the isomorphism of Theorem \ref{HH-Hopf}, the ring structure on $\mHH^\bullet(H)$ arises from the multiplication map $m:H\otimes H\lra H$ which is clearly $\ad$-equivariant. Thus, if
$z\in \zl_\Hig(H)$ and $y\in \mHH^{n}(H)$, then $z$ and $y$ are respectively represented by morphism $z:\Bbbk\lra P_0\subset H_\ad$, $y:\Bbbk\lra H_\ad[n] $, where $P_0$ is the injective envelope of $\Bbbk$. Their product cohomology class is then represented by
\[
zy: \Bbbk \cong \Bbbk\otimes \Bbbk\lra P_0\otimes H_\ad[n] \subset H_\ad\otimes H_\ad[n]\stackrel{m}{\lra} H_\ad[n].
\]
Since the tensor product $P_0\otimes H_\ad$ is an injective $H$ module, we have $\mH^n(H,P_0\otimes H_\ad)=0$ whenever $n>0$, and the result follows.

For the second statement, since the projective $SL_2(\Z)$-action by $\mathcal{F}$ and $\mathcal{L}$ preserves the module structure of $\ul_\ad$, it follows that $ P_0\otimes \zl_{\Hig}$ is an $SL_2(\Z)$-summand of $H_\ad$. The result then follows by taking Hopf cohomology $\mH^*(\ul,\mbox{-})$, using Theorem \ref{HH-Hopf} again.
\end{proof}

\paragraph{Higman ideal in case of the small quantum group.}
We want to specialize Proposition \ref{prop-Higman-in-HC} to the case of small quantum groups $\ul=\ul_q(\g)$. We will always assume $q$ is a root of unity satisfying the conditions of Section \ref{sec-q-group}.

\begin{theorem}   \label{thm-hig=intersection} 
Let $\ul=\ul_q(\g)$ be the small quantum group associated with a complex semisimple Lie algebra $\g$. Then
\[
\zl_{\Hig}(\ul)=\zl_{\HC}(\ul)\cap \mathcal{F}(\zl_{\HC}(\ul)).
\]
Furthermore, under the block decomposition \eqref{eqn-block-decomp} of $\ul=\prod_\lambda \ul_\lambda$, the Cartan matrix of each block $\ul_\lambda$ has rank one and $\mathrm{dim}(\zl_{\Hig}(\ul)\cap \ul_\lambda)=1$.
\end{theorem} 

\begin{proof} 
 It is known from \cite{BrGo} that the blocks of the Harish-Chandra center are isomorphic to algebras of coinvariants 
$S(\h)^{W_\lambda} /S(\h)^W_+$, where $W_\lambda \subset W$ is a subgroup stabilizing the weight $\lambda$ corresponding to the block $u_\lambda$. Clearly these are local Frobenius algebras, and thus each block has exactly a one-dimensional subspace that is annihilated by the radical of $\zl_{\HC}$. Therefore, by Corollary \ref{cor-annhilator-of-radical}, the dimension of $\zl_{\HC} \cap \mathcal{F}(\zl_{\HC})$ is less than or equal to the number of blocks of $\ul$. Denote this number by $b$. 

On the other hand, each block $\ul_\lambda$ has a non-zero Cartan matrix, and thus the rank of $C_H$ is greater than or equal to the number of blocks of $H$. Now part $(ii)$ of Proposition  \ref{prop-Higman-in-HC} (c.f.~Proposition \ref{prop-dim-proj-center})  implies that ${\rm dim}(\zl_\Hig(\ul))=\mathrm{rank}(C_H)$. We conclude, using part  $(iii)$ of Proposition \ref{prop-Higman-in-HC}, that
\[ 
b\leq \mathrm{dim}(\zl_{\rm Hig}(\ul))  \leq \mathrm{dim} (\zl_{\HC}(\ul) \cap  \mathcal{F}(\zl_{\HC}(\ul))) \leq b,
\]  
and equality must hold everywhere.

In particular, the rank of the Cartan matrix of each block of $\ul$ is precisely one. 
\end{proof} 

\begin{remark} 
Together with Proposition \ref{prop-Higman-in-HC}, Theorem \ref{thm-hig=intersection} implies that, for the small quantum group $\ul_q(\g)$, the intersection  
$\zl_{\HC}(\ul) \cap  \mathcal{F}(\zl_{\HC}(\ul))$ is spanned by $\jmath_l(\chi_P)$, where $P$ is a projective $\ul$-module:
\[ 
\zl_{\HC}(\ul) \cap  \mathcal{F}(\zl_{\HC}(\ul)) =\zl_{\rm Hig}(\ul) = \psi_l^{-1}(\pl_l(\ul)) = \jmath_l(\pl_l(\ul)). \] 
In particular, the Higman ideal $\zl_{\HC}(\ul) \cap  \mathcal{F}(\zl_{\HC}(\ul))$ is isomorphic to the ideal of the projective characters in the Grothendieck ring $G_0(\ul\dmod)\otimes_\Z \C$. 
\end{remark} 

\begin{corollary}
In the adjoint representation $\ul_q(\g)_\ad$, the projective cover $P(0)$ of the simple module $L(0)$ occurs with multiplicity equal to the number of blocks of $\ul_q(\g)$.
\end{corollary}
\begin{proof}
Follows from Proposition \ref{prop-adLambdaH} and Theorem \ref{thm-hig=intersection}.
\end{proof}

\begin{proposition}\label{prop-g-preserves-Hig}
The action of $U(\g)$ defined in Theorem  \ref{Action-div-powers} is trivial on the Higman ideal $\zl_\Hig(\ul)$. 
\end{proposition}
\begin{proof}
 The action of $U(\g)$ on $\zl$ is trivial on  $\zl_{\HC}$, as the Harish-Chandra center of $\ul$ also arises as the restriction of the Harich-Chandra center of the big quantum group to the small quantum group. Since $\zl_{\Hig} \subset \zl_{\HC}$  by  Theorem \ref{thm-hig=intersection}, the statement follows. 
 \end{proof}  
 %We use an alternative characterization of the Higman ideal in \cite[Theorem 5.2]{La}, which identifies, via Theorem \ref{thm-hig=intersection}, with
%\[
%\zl_{\Hig}(\ul)= \mathrm{Ann}(\mathrm{Rad}(\zl_\HC)),
%\]
%the annihilator of the radical of the Harish-Chandra center. By Lemma \ref{ad_on_ul}, the $\Ul(\g)$ action on $\zl$ integrates to a Lie group $G$ action by algebra automorphism of $\zl$, where $\mathrm{Lie}(G)=\g$. Furthermore, this action preserves the Harish-Chandra center, as the Harish-Chandra center of $\ul$ also arises as restriction of the Harich-Chandra center of the big quantum group to the small quantum group. It follows that $G$ acts trivially on $\zl_{\HC}$, $G$ also preserves the commutative algebra structure of $\zl$. Hence it preserves the algebraic properties of the $\zl_{\HC}$ and $\zl$ such as radical or annihilator. The first statement follows. 

%The last statement holds since $\zl_{\Hig}\subset \zl_{\HC}$, and $G$ acts trivially on the latter space.

\begin{example} 
Consider the principal block of the small quantum group $\ul_q(\g)$ for $\g = \s_3$. 
The block contains $6$ simple and $6$ projective modules. The Cartan matrix of the block has the entries $[P_i : L_j]$ for $1 \leq i, j \leq 6$. Recall that the $l$-restricted dominant weights in case $\g = \s_3$ include two open alcoves. 
Below $\{L_1. L_2. L_3 \}$ and $\{ P_1, P_2, P_3 \}$ denote the simple modules in the lower alcove of the $l$-restricted dominant weights and their projective covers. The modules $\{L_4, L_5, L_6 \}$ and $\{P_4, P_5, P_6 \}$ denote the simple modules in the second alcove and their projective covers. 
\[  
C_{0} =  
\left( \begin{array}{cccccc} 
24 & 24 & 24 & 12 & 12 & 12 \\ 
24 & 24 & 24 & 12 & 12 & 12 \\ 
24 & 24 & 24 & 12 & 12 & 12 \\ 
12 & 12 & 12 & 6 & 6 & 6 \\ 
12 & 12 & 12 & 6 & 6 & 6 \\ 
12 & 12 & 12 & 6 & 6 & 6 \\ 
\end{array} 
\right) .
\]  
Up to equivalence, there is a unique nontrivial singular block $\ul_{\mathrm{sing}}$. The simple module $L_1$ has its highest weight in the closure of the lower alcove, and $L_2, L_3$ in the closure of the upper alcove. 
\[
C_{\mathrm{sing}}=
\left(
\begin{matrix}
12 & 6 & 6\\
6 & 3 & 3 \\
6 & 3 & 3
\end{matrix}
\right)
\]
The Steinberg block is semisimple and, has the Cartan matrix $C_{\St}=(1)$. 

\begin{figure}[!htpb] 
\begin{center}
\caption{\small Orbits of $W \ltimes l P$ in the $l$-restricted weights for $\g= \s_3$, $l=7$.}  \label{pic_orbits} 
\vspace*{1mm} 
\begin{tabular}{ccc} 
\psfig{figure=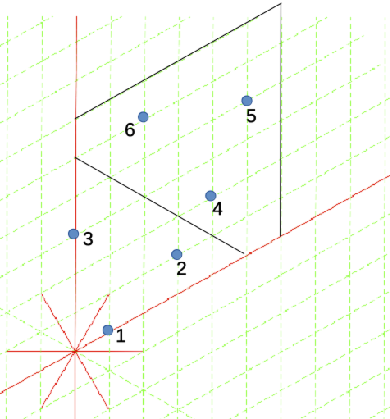, height=60mm}  
&  \;\;\;\;\;\; & 
\psfig{figure=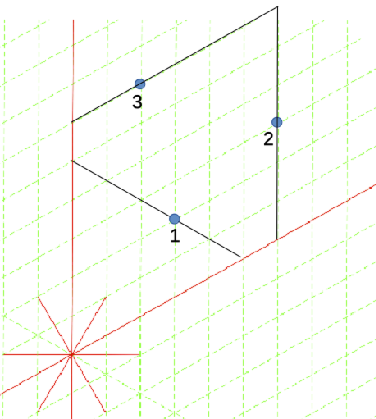, height = 60mm}  \\
{\small a regular orbit} & &  {\small a singular orbit}  
\end{tabular} 
\end{center}
\end{figure}

Figure \ref{pic_orbits} shows a regular and a nontrivial singular orbit for $\g = \s_3$. The $l$-restricted dominant weights are bounded by the red cone from below (the dominant chamber) and by the black boundary of the upper alcove from above. The weights on the black lines are the singular dominant $l$-restricted weights. The upper vertex corresponds to the Steinberg weight $(l-1)\rho$. The root system is shown in red. 
\end{example}

In case of $\g = \s_n$ we can give an explicit formula for the number of blocks of $\ul$, or equivalently, for the dimension of the Higman ideal in the center $\zl(\ul)$. 
\begin{theorem} \label{n_blocks}
The number of blocks for $\ul_q(\s_n)$ is equal to
\[
\frac{(l+n-1)\cdots (l+1)}{n!}=\frac{1}{l+n}{l+n \choose n}.
\]
Consequently, the Higman ideal for $\ul_q(\s_n)$ has the same dimension. 
%\[
%\mathrm{dim}_\C \zl_\Hig(\ul_q(\s_n))
%=\frac{1}{l+n}{l+n \choose n}.
%\]
\end{theorem}  
\begin{proof} 
Let $\Delta \subset Q \subset P$ denote the root system, the root lattice and the weight lattice of $\s_n$. We assume that $l$ greater than or equal to $n+1$ and  coprime to $n$. The linkage principle \cite{APW, APW2} implies that the number of blocks of $\ul_q(\s_n)$ equals to the number of the orbits of the {\it extended affine Weyl group} 
$W \ltimes l P$ in the set of $l$-restricted weights, or, equivalently, in the weight lattice. We can compute this number as follows. Let $\overline{A}$ denote the $l$-extended closed first dominant alcove: 
\[ \overline{A} = \{ \mu \in P^+ : \langle \mu , \check{\alpha} \rangle \leq l , \;\; \forall \alpha \in \Delta^+ \} .\] 
Then $\overline{A}$ is the fundamental domain for the action of the affine Weyl group $W \ltimes l Q$ on the weight lattice. 
Since $W \ltimes lQ$ acts simply transitively on the set of all alcoves, there is a one-to-one correspondence between the  integer weights in $\overline{A}$ and the orbits of $W \ltimes lQ$. To find the number of orbits of the extended affine Weyl group $W \ltimes lP$ in $\overline{A}$, we notice that $W \ltimes l P = (W \ltimes lQ) \ltimes \Omega$, where $\Omega \simeq P/Q$ is the subgroup of $W \ltimes l P$ stabilizing $\overline{A}$. Therefore, the number of orbits of $W \ltimes lP$ in $\overline{A}$ equals to the number of {\it root weights} in $\overline{A}$: $|\overline{A} \cap Q|$. By the work of Gorsky, Mazin and Vazirani  \cite[Theorem 3.4]{GMV}, 
this number is known to be equal to the rational Catalan number $c_{l,n} = \frac{1}{l+n}{l+n \choose n}$.   
\end{proof} 

\begin{remark} 
A similar argument in \cite{LQ2} shows that the number of regular blocks of $\ul_q(\s_n)$ equals to the rational Catalan number $c_{l-n,n} = \frac{1}{l}{l \choose n}$. 
\end{remark}

\paragraph{Geometric interpretation of Higman ideal.}Let $X_\lambda=G/P_\lambda$ be the (partial) flag variety of a complex semisimple Lie group, and $\Nt_\lambda=T^*X_\lambda$ be the Springer variety. Denote by ${\rm pr}:\Nt_\lambda\lra X_\lambda$ the natural cotangent projection map. Set $m=\mathrm{dim}_\C(X_\lambda)$. Since ${\rm pr}$ is an affine map, the coherent-cohomological dimension of $\Nt_\lambda$ is equal to $m$. 

Following \cite{LQ1,LQ2}, it will be beneficial to arrange the Hochschild cohomology ring of $\Nt$ into the direct sum of two tables, accounting for the even and old cohomology groups. 

%\begin{equation}\label{eqn-canonical-filtration}
%0 \lra {\rm pr}^* \Omega_{X}  \lra T\Nt \lra {\rm pr}^*  TX \lra 0.
%\end{equation}

\begin{equation}
\begin{array}{|c||c|c|c|c|} 
\hline
{\scriptstyle \mH^i(\wedge^j T\Nt_\lambda)}   &  {\scriptstyle j-i=0}   &  {\scriptstyle j-i=2} & \dots & {\scriptstyle j-i=m} \\[1mm]  \hline \hline
{\scriptstyle i+j = 0}    &  {\scriptstyle \C[\Nc_\lambda]}      & 0  & \dots & 0 \\[1mm]  \hline 
{\scriptstyle i+j = 2}    &   \ast       &   \ast & \dots & 0 \\[1mm]  \hline
\vdots    &   \ast      &  \ast  & \dots &  0 \\[1mm]  \hline
{\scriptstyle i+j = 2m}    &   \ast      & \ast  &  \dots  & {\scriptstyle \C[\Nc_\lambda]\{-2m\} }\\[1mm]  \hline
  \end{array} 
\oplus
\begin{array}{|c||c|c|c|c|} 
\hline
{\scriptstyle \mH^i(\wedge^j T\Nt_\lambda)}   &  {\scriptstyle j-i=1}   &  {\scriptstyle j-i=3} & \dots & {\scriptstyle j-i=2m-1} \\[1mm]  \hline \hline
{\scriptstyle i+j = 1}    &  \ast    & 0  & \dots & 0 \\[1mm]  \hline 
{\scriptstyle i+j = 3}    &   \ast       &   \ast  & \dots & 0 \\[1mm]  \hline
\vdots    &   \ast       &  \ast & \dots &  0\\[1mm]  \hline
{\scriptstyle i+j = 2m-1}    &   \ast     &  \ast  &  \dots  & \ast \\[1mm]  \hline
  \end{array}  
\end{equation}

One reason to exhibit the Hochschild cohomology table this way is to exhibit an apparent $\s_2(\C)$ action along the Northwest-Southeast diagonals of the table. The action is generated by wedging with the Poisson bivector field and contraction with the holomorphic symplectic form on $\Nt_\lambda$. We refer the reader to \cite[Theorem 4.3]{LQ1}, \cite[Theorem 2.11]{LQ2} for the details. 

Because of this $\s_2(\C)$ action, one deduces that the main diagonal entries in the even table contains copies of $\C[\Nc_\lambda]\{-2k\}$, $k=0,\dots, -2m$. The degree shift arises from the fact that the Poisson bivector field on $\Nt_\lambda$ has degree $-2$.

As follows from Theorem \ref{thm-hig=intersection},  The Cartan matrix of a block $\ul_\lambda$ of the small quantum group $\ul_q(\g)$ has rank one. Consequently, the projective center of $\ul_\lambda$ has dimension one.

Geometrically, we may identify the position of the projective center:

\begin{proposition}\label{prop-geometric-Higman-ideal}
There is an isomorphism of Hochschild cohomology group
\[
\mH^m(\Nt_\lambda,\wedge^m T\Nt_\lambda)^\bullet \cong \C\{-2m\}.
\]
\end{proposition}
\begin{proof}
The coherent sheaf ${\rm pr}_* (\wedge^m T\Nt_\lambda)$ has a natural filtration whose subquotients consist of coherent sheaves of the form $\wedge^r \Omega_X \otimes \wedge^{m-r}TX \otimes S^\bullet (TX)$. Then, by Serre duality, we have
\begin{align*}
\mH^m(X, \wedge^r \Omega_X \otimes \wedge^{m-r}TX \otimes S^k (TX)) & \cong  \mH^0 ( X, \wedge^{r}TX \otimes \wedge^{m-r} \Omega_X \otimes S^k(\Omega_X)\otimes \Ox(K))^*\\
& \cong \mH^0(X,\wedge^{m-r}\Omega_X\otimes \wedge^{m-r} \Omega_X\otimes S^k(\Omega_X) )^*\\
& \cong
\left\{
\begin{array}{cc}
\C\{-2m\} & r=m,~k=0,\\
0 & \textrm{otherwise}.
\end{array}
\right. 
\end{align*}
It follows that, by an easy induction on the filtration, we have
\[
\mH^m(\Nt_\lambda,\wedge^m T\Nt_\lambda)^k \cong
\left\{
\begin{array}{cc}
 \C\{-2m\} & k=-2m,\\
 0 & k\neq -2m.
\end{array}
 \right.
\]
The proposition follows.
\end{proof}

 This one-dimensional subspace in Proposition \ref{prop-geometric-Higman-ideal} is in the Higman ideal: it is in the Harish-Chandra center and annihilates the radical of the block (or in the Fourier transform of $\zl_{\HC}$.

We can use Theorem \ref{BL} and the techniques developed in \cite{LQ1}, \cite{LQ2} to compute higher Hochschild cohomologies for $\g = \s_3$. In particular, we have the following result.

\begin{example} Let $V(0)$  and $V(\rho)$ denote the trivial and the adjoint representation of $\s_3$, respectively. Then we have an isomorphism of the bigraded vector spaces 
\[
\mHH^1(\ul_0(\s_3))\cong \bigoplus_{i+j+k=1} \mH^i(\Nt,\wedge^j T\Nt)^k
\]
with the bigraded components isomorphic to $\s_3$-modules as shown in the following table 
\[ 
\begin{array}{|c||c|c|c|} 
\hline
i+j = 1   & V(\rho) \oplus V(0) & & \\[1mm]  \hline 
i+j =3   & V(\rho)^{\oplus 2} \oplus V(0)^{\oplus 3} & V(\rho) \oplus  V(0)  & \\[1mm]  \hline
i+j = 5  & V(0)^{\oplus 2} &  V(\rho)^{\oplus 2} \oplus V(0)^{\oplus 3} & V(\rho) \oplus V(0)  \\[1mm] \hline   \hline 
             & j-i =1 & j-i =3 & j-i = 5  \\  \hline 
\end{array}
\] 
\end{example}

\section{Derived center of the small quantum $\mathfrak{sl}_2$ }    \label{section_sl2} 
\paragraph{The center of small quantum $\mathfrak{sl}_2$.}
We first recall the following description of the center of $\ul(\mathfrak{sl}_2)$ by Kerler \cite{Ker} (see also \cite{FGST}).

\begin{theorem}\label{thmsl2center}
At a primitive odd $l$th root of unity $q$, the center $\zl(\mathfrak{sl}_2)$ of $\ul_q(\mathfrak{sl}_2)$ is isomorphic to the commutative algebra
\[
\zl(\mathfrak{sl}_2)\cong \prod_{i=1}^{\frac{l-1}{2}} \dfrac{\C[x_i,y_i]}{(x_i^2,x_iy_i,y_i^2)}\times \C
\]
\end{theorem}

The last factor of $\C$ corresponds to the center of the semisimple Steinberg block. We will denote by $e_i$ the central idempotent in $\ul(\mathfrak{sl}_2)$ that corresponds to the element $(0,\dots, 1, \dots, 0)$, where the only nonzero entry appears in the $i$th position. The last central idempotent  corresponds to the Steinberg block and  will be denoted by $e_{\St}$.

The generators $x_i,y_i$ for the regular blocks can be identified, via the Radford maps, as follows.

Define the shifted trace-like functionals
\[
\chi_i(h):=\mathrm{Tr}\big|_{L(i)}(Kh), \quad i=0,\dots, l-1.
\]
Then each $\chi_i\in \cl_l(\mathfrak{sl}_2)$ since $S^2=\ad K$:
\begin{align}
\chi_i(ab)=\mathrm{Tr}\big|_{L(i)}(Kab)=\mathrm{Tr}\big|_{L(i)}(KaK^{-1}Kb)=
\mathrm{Tr}\big|_{L(i)}(KbKaK^{-1})=\chi_i(bS^2(a)).
\end{align}
According to  \cite{FGST}, we can define the variables $\{ x_i, y_i \}_{i=1}^{\frac{l-1}{2}}$ and $e_{\St}$  in Theorem \ref{thmsl2center}  via the Radford isomorphism (Theorem \ref{thmRadfordiso}) as follows: 
\begin{subequations}
\begin{equation}
x_{i+1}=\psi_l^{-1}(\chi_{i}), \quad y_{i+1}=\psi_l^{-1}(\chi_{l-2-i}) \quad (i=0, \dots, \frac{l-3}{2}).
\end{equation}
and
\begin{equation}
e_{\St}=\frac{1}{l\sqrt{l}}\psi_l^{-1}(\chi_{l-1}).
\end{equation}
\end{subequations}

\paragraph{The adjoint representation.} 
We will use the following result of Ostrik to obtain the total Hochschild cohomology of the small quantum group $\ul_q(\s_2)$. 

\begin{theorem}  \cite{Ostrikadjoint}    \label{ostrik-ad} 
Let $l \geq 3$ be an odd integer. Let $L(m)$ denote the simple $\ul_q(\mathfrak{sl})$-module of highest weight $m$ such that $0 \leq m \leq l-1$, and $P(m)$ its projective cover. Let $\phi^*(V(1))$ denote the $\Ul$-module obtained as the Frobenius pullback of the simple $\s_2$-module of highest weight $1$. Then the small quantum group $\ul(\s_2)$ has the following decomposition with respect to the Hopf-adjoint action by $\Ul_q(\s_2)$: 
\begin{align*}
\ul_q(\s_2)_{\rm ad}  \cong & L(l-1)^{\oplus l}  \bigoplus  \left( \bigoplus_{i=0}^{\frac{l-3}{2}} P(2i)^{\oplus (\frac{l+1}{2}+i)} \right) \bigoplus
\left(\bigoplus_{i=0}^{\frac{l-3}{2}} L(2i)^{\oplus (l-1-2i)} \right)\\
& \bigoplus 
\left(\bigoplus_{i=0}^{\frac{l-3}{2}} ( \phi^* V(1) \otimes L(l -2 - 2i))^{\oplus (\frac{l-1}{2}-i)} \right) .
\end{align*}
\end{theorem}

\begin{corollary}   \label{u-sl2-ad} 
The submodule in $\ul_q(\s_2)_{\rm ad}$ with composition factors in the principal block has the following decomposition: 
\[ 
\left( \ul_q(\s_2)_{\rm ad}\right)_0 \cong P(0)^{\oplus \frac{l+1}{2}} \oplus L(0)^{\oplus (l-1)} \oplus (\phi^*V(1) \otimes L(l-2))^{\oplus \frac{l-1}{2}} . 
\] 
\end{corollary} 

It follows from Ostrik's theorem, together with Lemma \ref{ad-H-on-itself}, Proposition \ref{prop-left-trace-functionals} that the dimension of the center and shifted trace-like functionals are equal to
\begin{subequations}
\begin{equation}
\mathrm{dim}_\C(\zl(\mathfrak{sl}_2))=\mathrm{dim}(\Hom_{\ul}(L(0),\ul_\ad))=\frac{l+1}{2}+l-1=\frac{3l-1}{2},
\end{equation}
\begin{equation}
\mathrm{dim}_\C(\cl_l(\mathfrak{sl}_2))=\mathrm{dim}(\Hom_{\ul}(\ul_\ad,L(0)))=\frac{l+1}{2}+l-1=\frac{3l-1}{2},
\end{equation}
\end{subequations}
since $L(0)$ occurs with multiplicity two inside $P(0)$, once as a subrepresentation and once as its head. We would then like to analyze how Kerler's explicit central elements are contained in projective or simple summands of $\ul(\mathfrak{sl}_2)_\ad$.

\begin{theorem}\label{thm-sl2-central-elements-in-adrep}
For $\ul_q(\mathfrak{sl}_2)$ at a primitive odd $l$th root of unity $q$, the following results holds.
\begin{enumerate}
\item[(i)] $\mathrm{dim}(\zl_\Hig(\ul_q(\s_2))=\frac{l+1}{2}$. Moreover, $\zl_{\Hig}(\ul_q(\s_2))  = {\mathrm{Soc}} \left( P(0)^{\oplus \frac{l+1}{2}} \right)$.  
%consists of central elements of $\ul(\mathfrak{sl}_2))$ that are contained in projective summands of $\ul(\mathfrak{sl}_2)_\ad$.
\item[(ii)] The central idempotents $e_i$, $i=1,\dots, \frac{l-1}{2}$ and the central elements $x_i$, $i=1,\dots,\frac{l-1}{2}$, are contained in the direct sum of trivial summands of $\ul(\mathfrak{sl}_2)_\ad$. 
\item[(iii)] The central elements $x_i+y_i$, $i=1,\dots,\frac{l-1}{2}$, and the Steinberg idempotent $e_{\St}$ span the subspace  $ {\mathrm{Soc}} \left( P(0)^{\oplus \frac{l+1}{2}} \right)$. 
\item[(iv)] The integral $\Lambda$  is contained in a trivial summand.
\end{enumerate}
\end{theorem}
\begin{proof}
The first statement follows from Proposition \ref{prop-adLambdaH} and Proposition \ref{prop-dim-proj-center} by computing the Cartan matrix for each regular block to be
\[
\left(
\begin{matrix}
2 & 2 \\
2 & 2
\end{matrix}
\right),
\]
while the Steinberg block has its Cartan matrix $(1)$.

Using Proposition \ref{prop-left-trace-functionals}, each shifted trace-like functional in $f\in\cl_l$ defines a $\ul$-module map
\[
f:\ul_\ad\lra \C.
\]
If $z$ is a central element and $f\in \cl_l$ satisfies
 $f(z)\neq 0$, then the composition map
\begin{equation}
\C \stackrel{\iota_z}{\lra} \ul_\ad \stackrel{f}{\lra} \C, \quad 1\mapsto f(z)
\end{equation}
provides the inclusion $\iota_z :\C\lra \ul_\ad$, $1\mapsto z$ with a splitting, so that $\C z\cong L(0)$ is a direct summand of $\ul_\ad$.

Next, we observe that, when decomposing 
$$\ul=\prod_{i=1}^{\frac{l-1}{2}}\ul e_i \times \ul e_{\rm St}$$
into a direct product of indecomposable two-sided ideals (blocks), the adjoint representation preserves the two sided-ideals. Here the first $\frac{l-1}{2}$ factors correspond to regular blocks, while the last one stands for the Steinberg block. Each $\ul e_i$ must contribute the same number of $P(0)$'s and $L(0)$'s into $\ul_\ad$, for, otherwise, the regular blocks would have non-isomorphic Hochschild cohomology groups, violating the translation principle. Therefore, we may restrict our attention to each block, and compute the effect of $\cl_l$ restricted to the center of each block.

For the functions $\chi_j$,  $j \in \{0,1,\dots, l-1\}$, we have that, when evaluated on the central idempotents $e_i$, $i=1,\dots, \frac{l-1}{2}$,
\[
\chi_j(e_i)=\left\{
\begin{array}{cc}
\dfrac{q^{j+1}-q^{-j-1}}{q-q^{-1}}=\mathrm{dim}_q(L(j)) & j=i-1,~ l-i-1 , \\
& \\
0 & j\neq i-1,~ l-i-1 .
\end{array}
\right.
\]
When they are evaluated on $x_i$, $y_i$, $i=1,\dots, \frac{l-1}{2}$, we have
\[
\chi_j(x_i)=\chi_j(y_i)=0.
\]

For the functional $\lambda_l$, we compute that
\[
\lambda_l(x_i)=\lambda_l(x_i1)=\chi_{i-1}(1)=\mathrm{dim}_q(L(i-1))=[i-1]_q,
\]
\[
\lambda_l(y_i)=\lambda_l(y_i1)=\chi_{l-i-1}(1)=\mathrm{dim}_q(L(l-i-1))=[l-i-1]_q= -[i-1]_q .
\]
\[
\lambda_l(e_\St)=\lambda_l(e_\St1)=\chi_\St(1)=\mathrm{dim}_q(L(l-1))=0 .
\]
It follows that, in the regular blocks, the product homomorphism 
\[
\chi_{i-1} \times \chi_{l-i-2} \times \lambda_l : \C\langle e_i,x_i,y_i \rangle\lra \C^3
\]
has two-dimensional image, and $\C \langle x_i+y_i \rangle$ lies in the kernel of this map. It follows that, a complementary subspace to the kernel, say either $\C\langle e_i, x_i \rangle$ or $\C\langle e_i, y_i \rangle$, constitutes a trivial summand in $(\ul e_i)_{\ad}$.

 Denote by $\C\langle e_i, x_i \rangle_{i=1}^{\frac{l-1}{2}}$ the central subalgebra spanned by the elements in the bracket, then
\[
\mathrm{dim}\C\left\langle e_i, x_i \right\rangle_{i=1}^{\frac{l-1}{2}}=l-1,
\]
which is already equal to the number of trivial summands in $\ul_\ad$. It follows that the complementary subspace in the center agrees with $\ad \Lambda(\ul)$:
\[
\C \left\langle e_{\mathrm{St}}, x_i+y_i \right\rangle _{i=1}^{\frac{l-1}{2}}=\bigcap_{i=1}^{\frac{l-1}{2}}\mathrm{Ker}(\chi_i\big|_{\zl})\bigcap \mathrm{Ker}(\lambda_l\big|_{\zl})=\zl_{\Hig}.
\]

The last statement follows from the fact that $\lambda_l(\Lambda)=1$. This finishes the proof of the theorem.
\end{proof}

\paragraph{The $\ul_q(\mathfrak{sl}_2)$ derived center.}
Combined with the result \cite{GinKu} on the cohomology of the small quantum group, this decomposition allows us to derive 
the total Hochschild cohomology of the small quantum group $\ul(\s_2)$. 

Denote by $\C[\Nc]_+$ the augmentation ideal generated by all positive degree elements in the graded ring $\C[\Nc]$.
We will consider separately the odd and even degree of the Hochschild cohomology. 

First we compute the odd Hochschild cohomology part of $\ul_q(\mathfrak{sl}_2)$. In order to do so, we need to compute the cohomology of the module $\phi^*(V(1))\otimes L(l-2)$ appearing as direct summands of $\ul(\mathfrak{sl}_2)$. Although the factor $\phi^*(V(1))$ restricts to a trivial module over $\ul_q(\s_2)$, it carries a nontrivial action by $\Ul_q(\s_2)\qq \ul(\s_2)=U(\s_2)$.

\begin{proposition}\label{prop-HH-sl2-odd}
There are isomorphism of $U(\s_2)$-modules
\[
\mathrm{Ext}^i_{\ul(\s_2)} (L(0), \phi^*(V(1))\otimes L(l-2) ) \cong
\left\{
\begin{array}{cc}
V(i-1)\oplus V(i+1) & i~\textrm{odd},\\
0 & i~\textrm{even}.
\end{array}
\right.
\]
Consequently, for the odd degrees, there is an isomorphism of graded $ \C[\Nc]\rtimes U(\s_2) $-modules
\[
\mathrm{Ext}^{2\bullet+1}_{\ul_q(\s_2)}(L(0),\phi^*(V(1))\otimes L(l-2)) \cong \C[\Nc][-1]\oplus \C[\Nc]_+[1].
\]
\end{proposition}

\begin{proof}
Regarding $L(l-2)$ as a (tilting) module over $\Ul=\Ul_q(\s_2)$, we will construct an injective resolution for $L(l-2)$ over $\Ul$. The resolution would then restricts to an injective resolution $L(l-2)$ over $\ul=\ul_q(\s_2)$ (see e.g. \cite{APW, APW2, Andsntensor}).

The module $L(l-2)=T(l-2)$ is a tilting module over $\Ul$. There is a resolution of tilting modules 
\begin{equation}
0\lra T(l-2) \lra T(l) \stackrel{d_0}{\lra} T(3l-2) \stackrel{d_1}{\lra} T(3l) \stackrel{d_2}{\lra} T(5l-2) \stackrel{d_3}{\lra} \dots
\end{equation}
The differentials ($k=0,1,\dots,$)
\begin{equation}
\cdots \lra T((2k+1)l-2) \xrightarrow{d_{2k}} T((2k+1)l) \xrightarrow{d_{2k+1}} T((2k+2)l-2) \lra \cdots
\end{equation}
are given by the unique maps, up to rescaling, of  tilting modules whose neighboring highest weights lie in the same orbit of $l-2$ by the affine Weyl group action on the weight lattice.

Lusztig's tensor product decomposition theorem implies that, in this case, there is a tensor product decomposition
\[
T((2k+1)l)\cong \phi^*V(2k)\otimes T(l), \quad \quad 
T((2k+1)l-2) \cong \phi^* V(2k-1)\otimes T(2l-2).
\]
Here $V(m)$ denotes the simple $U(\s_2)$-module of highest weight $m \geq 0$ and dimension $m+1$.

Restrict to the small quantum group, and keeping track of the $U(\s_2)$ action at the same time, we obtain an injective resolution
\[
0 \to L(l-2) \to P(l-2) \to \phi^*V(1) \otimes P(0) \to \phi^*V(2) \otimes P(l-2) \to \phi^*V(3) \otimes P(0) \to \dots ,
\]   
where we have used that
\[
T(l)\big|_{\ul}\cong P(l-2), \quad \quad T(2l-2)\big|_{\ul}\cong P(0). 
\]
The Frobenius pullbacks are flat over $\ul$. Then to compute ${\rm Ext}^\bullet_\ul (L(0), \phi^*V(1) \otimes L(l-2))$ we take the tensor product of the above resolution with $\phi^*(V(1))$ to obtain
\begin{align}\label{eqn-resolution1}
0 &\to \phi^*V(1)\otimes L(l-2)  \to \phi^*V(1)\otimes P(l-2) \to \phi^*(V(0) \oplus V(2)) \otimes P(0) \to
 \nonumber \\
 & \to \phi^*(V(1)\oplus V(3)) \otimes P(l-2)  \to \phi^*(V(2) \oplus V(4)) \otimes P(0) \to \ldots 
 \end{align} 
 
Equivalently, this resolution can be obtained from the injective resolution for $L(0)$ constructed in a similar fashion as above:
\begin{equation}\label{eqn-resolution2}
0 \to L(0) \to P(0) \to \phi^*V(1) \otimes P(l-2) \to \phi^*V(2) \otimes P(0) \to \phi^*V(3) \otimes P(l-2) \to \ldots 
\end{equation}
by tensoring it with the module $\phi^*V(1)\otimes L(l-2)$.

Note that $\Hom_{\ul}(L(i),P(j))=\delta_{i,j} \C$ for $i,j\in \{0,l-2\}$. Using this and resolution \eqref{eqn-resolution2} we obtain that 
\begin{equation}\label{eqn-ext-L0}
\C^\bullet[\Nc]\cong{\rm Ext}^{\bullet}_\ul (L(0), L(0)) \cong \bigoplus_{k\in \N} V(2k) ,
\end{equation}
where $V(2k)$ sits in homological degree $2k$. Likewise, using resolution \eqref{eqn-resolution1}, we have
\begin{equation}\label{eqn-ext-L0Ll-2}
\mathrm{Ext}^\bullet (L(0), \phi^*V(1)\otimes L(l-2))= \bigoplus_{k\in \N} \left( V(2k)\oplus V(2k+2) \right),
\end{equation}
where the summand $ V(2k)\oplus V(2k+2) $ appears in homological degree $2k+1$. Comparing the expressions \eqref{eqn-ext-L0} and \eqref{eqn-ext-L0Ll-2}, the Proposition follows. 
\end{proof}

\begin{proposition} \label{prop-HH-sl2-even}
The even part of the Hochschild cohomology ring of $\ul_q(\mathfrak{sl}_2)$ is isomorphic to the graded algebra
\[
\mHH^{2\bullet} (\ul_q(\mathfrak{sl}_2))
\cong
\dfrac{\C[\Nc]\otimes \zl(\mathfrak{sl}_2) }{\left(\C[\Nc]_+\otimes (x_i+y_i),\C[\Nc]_+ \otimes e_{\mathrm{st}}|i=1,\dots, \frac{l-1}{2}\right)}.
\]
Here the terms in the denominator stands for the ideal generated by the corresponding elements in the tensor product commutative algebra.
\end{proposition}
\begin{proof}
We use Theorem \ref{HH-Hopf} and Corollary \ref{u-sl2-ad}. 
The contribution from $L(0)$ is obtained by using the injective resolution \ref{eqn-resolution2} and results in 
the term \ref{eqn-ext-L0}. The contribution from the injective module $P(0)$ is contained entirely in the zeroth cohomology. 
Comparing with Theorem \ref{thm-sl2-central-elements-in-adrep}, we see that $\C(x_i+y_i)$ and $\C(e_{\mathrm{st}})$ are contained in projective summands of $\ul_\ad$.  Each of their contribution, say, the term $\C(x_i+y_i)$, to the entire Hochschild cohomology can be computed to be
\[
\mathrm{Ext}^\bullet_{\ul}(L(0),P(0))\cong \Hom_{\ul}(L(0),P(0)) \cong \C \cong \dfrac{\C[\Nc]\otimes \C(x_i+y_i)}{\C[\Nc]_+\otimes \C(x_i+y_i)}.
\]
Here the first isomorphism holds since $P(0)$ is also injective over $\ul$, so that there are no higher ext groups.
The term $(\phi^* V(1) \otimes L(l-2))$ does not contribute to even-degree cohomology by Proposition \ref{prop-HH-sl2-odd}. 
The proposition follows.
\end{proof}

Combining Proposition \ref{prop-HH-sl2-even} and \ref{prop-HH-sl2-odd}, we immediately obtain

\begin{theorem}\label{thm-HH-sl2}
The total Hochschild cohomology group of $\ul_q(\s_2)$, where $q$ is an odd primitive $l$th root of unity, is isomorphic to the graded $\C[\Nc]\rtimes U(\s_2)$-module
\[
\mHH^\bullet(\ul_q(\s_2))\cong
\prod_{i=1}^{\frac{l-1}{2}}
\left(
\C[\Nc]e_i\oplus  \C[\Nc]\mu_i \oplus \C[\Nc]_+\nu_i
\oplus \dfrac{\C[\Nc]x_i \oplus\C[\Nc]y_i}{\C[\Nc]_+(x_i+y_i)} \right)\times \C e_{\rm st},
\]
where the module generators $e_i,~\mu_i,~\nu_i,~x_i,~y_i,~e_{\rm st}$ have homological degrees
 $$\mathrm{deg}(e_i)=\mathrm{deg}(e_{\rm st})=\mathrm{deg}(x_i)=\mathrm{deg}(y_i)=0, \quad \quad \mathrm{deg}(\mu_i)=1, \quad \quad \mathrm{deg}(\nu_i)=-1. $$
Furthermore, the product structure respects the block decomposition of $\ul_q(\s_2)$. 
\end{theorem} 
\begin{proof}
Only the last part needs some remark. The Steinberg block is semisimple and does not have any higher Hochschild cohomology. On the other hand, the regular blocks $\ul_\lambda$ are all equivalent to the principal block $\ul_0$ by Jantzen's translation principle, and therefore $\phi^*V(1)\otimes L(l-2)$ must appear exactly once in $(\ul_\lambda)_\ad$, and the total odd Hochschild cohomology groups are evenly distributed among the $\frac{l-1}{2}$ regular blocks. 
\end{proof}

However, Theorem \ref{thm-HH-sl2} does not provide the algebra structure of the Hochschild cohomology ring yet. This will be done in the next subsection.

\paragraph{Geometric construction and algebra structure.} 
In the final part of this section, we use Theorem \ref{BL} to compare the Hochschild cohomology of each block with with the algebraic computations above, and determine the algebra structure of the Hochschild cohomology ring.

\begin{proposition}  \label{HH_sl2_0} 
The $\C^*$-equivariant Hochschild cohomology of the Springer variety is isomorphic to 
\[ 
 \mHH^s_{\C^*} (\Nt)=\left\{ \begin{array}{ll} 
V(0)^{\oplus 3} , & s =0 ,\\ 
   V(s)^{\oplus 2},  & s \geq 2, \;\; s \; {\rm even} \\
    V(s-1) \oplus V(s+1), & s \geq 1, \;\; s \; {\rm odd} \\ 
  \end{array} \right.  \]  
\end{proposition} 

\begin{proof} 
Denote for simplicity the $SL_2(\C)$ flag variety by $\PS=SL_2(\C)/B$, where $B$ is a Borel subgroup, and $\Nt=T^*\PS$ is the Springer variety.

By the Hochschild-Rosenberg-Kostant Theorem, there is an isomorphism of graded vector spaces
$$\mHH^\bullet_{\C^*}(\Nt) \cong \bigoplus_{i+j+k=\bullet} \mH^i(\Nt, \wedge^j T\Nt)^k$$
We compute each cohomology group $\mH^i(\Nt, \wedge^j T\Nt)^k$ by pushing forward the poly-tangent bundles along the projection   ${\rm pr} : \Nt=T^*\PS \to  \PS$ and making use of the short exact sequence  
$$0 \lra {\rm pr}^* \Omega_\PS  \lra T\Nt \lra {\rm pr}^*  T\PS \lra 0.$$  

Below, we identify $T\PS \cong \Ox(2)$ and $T^*\PS=\Omega_\PS \cong \Ox(-2)$.            
It is important to keep track of the $\C^*$-degree of the sheaves: $\Omega_\PS$ has $\C^*$-degree $-2$, $T\PS$ has $\C^*$-degree $0$. The pushforward 
$${\rm pr} (\Ox_\Nt ) \cong  SL_2(\C)\times_B\C[\n] \cong  \bigoplus_{t \geq 0} \Ox(2t)^{k=2t}$$
splits into $\C^*$-homogeneous summands indicated as superscripts. Here $\n$ is the nilpotent radical of the Lie algebra of $B$.

We will also use the Borel-Bott-Weil Theorem for $SL_2(\C)$: 
\[ 
\mH^i (\PS,  \Ox(m) ) 
\cong \left\{ 
\begin{array}{cc} 
 V(m)    & i =0, ~ m \geq 0 , \\ 
 V(-m-2) & i = 1, ~ m \leq -2 , \\
    0    & {\rm otherwise.}     
\end{array}   
\right. 
\]                                           
   
We divide the cases according the wedge power $\wedge^jT\Nt$ of $j$. Since $T\Nt$ has rank two, $j$ ranges over $0,~1,~2$.
                                                
Let $j=0$. Then we have ${\rm pr}_* (\wedge^0T\Nt)={\rm pr}_*(\Ox_\Nt)\cong 
\oplus_{t \geq 0} \Ox(2t)$.  
We get the contribution of this component to the total Hochschild cohomology is equal to
\[
\mH^0(\Nt,\wedge^0T\Nt)^{2t}\cong  V(2t) \quad (t\in \N).
\] 

Let $j=1$. 
Pushing forward the short exact sequence of sheaves $0 \to {\rm pr}^* \Ox(-2) \to T\Nt \to {\rm pr}^* \Ox(2) \to 0$  to $\PS$ gives us a short exact sequence of bundles of infinite rank 
\[
0\lra \bigoplus_{t=0}^\infty \Ox(-2+2t)^{2t-2} \lra \bigoplus_{t=0}^\infty \left({\rm pr}_* T\Nt\right)^{2t-2} \lra \bigoplus_{t=0}^\infty \Ox(2+2t)^{2t}\lra 0,
\]
since ${\rm pr}:\Nt\lra \PS$ is an affine morphism. Taking cohomology, we obtain, as part of a long exact sequence, the short exact sequences
\[ 
0 \to \mH^0(\PS, \Ox(2t-2)) \to \mH^0(\Nt, T\Nt)^{2(t-1)} \to \mH^0(\PS, \Ox(2t)) \to 0 ,
\] 
\[ 
0 \to \mH^1(\PS, \Ox(-2)) \to \mH^1(\Nt, T\Nt)^{-2} \to 0 \to 0 .
\] 
Here we have used that $\mH^1(\PS, \Ox(2t-2) ) =0$  for all $t \geq 1$.                                             
Therefore we have:
\[  
\mH^0(\Nt, T\Nt)^{2t-2}\cong  V(2t-2) \oplus V(2t) , \quad (t \geq 1) \quad \quad \mH^1(\Nt, T\Nt)^{-2}\cong  V(0)
\]

Let $j=2$. Then we have $\wedge^2T\Nt\cong \Ox_\Nt$, but with the $\C^*$-degree shifted down by $-2$. Therefore, we have, as in the $j=0$ case, the nonzero contribution to Hochschild cohomology comes from
\[
\mH^0(\Nt, \wedge^2T\Nt)^{2t-2}\cong V(2t) \quad (t\in \N),
\]
and vanishes otherwise. 
The Proposition follows. 
\end{proof}

Comparing the injective resolution for $L(0)$ over $\ul(\s_2)$, given in the proof of Proposition \ref{prop-HH-sl2-odd}, with the result ${\rm Ext}^{2\bullet}_\ul (L(0), L(0)) \cong \C^\bullet[\Nc]$, we find $\C^s[\Nc] \cong V(2s)$ for all $s \geq 0$. 

Perhaps it is more illustrative to exhibit the computation results of Proposition \ref{HH_sl2_0} into the following table:
\begin{equation}\label{table-sl2}
\begin{array}{|c||c|c|} 
\hline
{\scriptstyle \mH^i(\Nt,\wedge^j T\Nt)}   &  {\scriptstyle j-i=0}   &  {\scriptstyle j-i=2} \\[1mm]  \hline \hline
{\scriptstyle i+j = 0}    &  {\scriptstyle \C[\Nc]}      & 0 \\[1mm]  \hline 
{\scriptstyle i+j = 2}    &   {\scriptstyle \C \{-2\} }        &  {\scriptstyle \C[\Nc]\{-2\}}\\[1mm]  \hline
  \end{array} 
\oplus
\begin{array}{|c||c|c|} 
\hline
{\scriptstyle \mH^i(\Nt,\wedge^j T\Nt)}   &  {\scriptstyle j-i=-1}   &  {\scriptstyle j-i=1} \\[1mm]  \hline \hline
{\scriptstyle i+j = 1}    &  0      & {\scriptscriptstyle \C[\Nc][-1]\oplus  \C[\Nc]_+[1]} \\[1mm]  \hline 
{\scriptstyle i+j = 3}    &  0           &  0 \\[1mm]  \hline
  \end{array} 
\end{equation}
The table splits the total Hochschild cohomology of $\ul_0(\s_2)$ into a direct sum of even and odd parts, and the notation $\{-2\}$ stands for a $\C^*$-degree shift down by 2.

This shows that Theorem \ref{thm-HH-sl2} agrees with the result for the principal block given in Proposition \ref{HH_sl2_0}. Indeed, 
taking product over the $\frac{l-1}{2}$ regular blocks, which are Morita equivalent to the principal block, of the above table, as well as a copy of $\C$ for the Steinberg block, one recovers the statement of Theorem \ref{thm-HH-sl2}. 

Furthermore, the table exhibits a triply-graded module structure of $\mHH^\bullet(\ul(\s_2))$ over the graded ring $\C[\Nc]\rtimes U(\s_2)$. Namely, the $i$, $j$ and $k$ (or $\C^*$) gradings in $\mH^{i}(\Nt,\wedge^jT\Nt)^k$. The Hochschild cohomological grading is equal to $i+j+k$. 
 In general the Hochschild-Rosenberg-Kostant Theorem is only an isomorphism of algebras after twisting by the square root of Todd class (see, for instance, \cite{CaVdB}). We expect the Hochschild cohomology ring structure of $\ul(\s_2)$ to only inherit the $\C^*$-grading from this table.  But for $\Nt=T^*\PS$, the Todd class is trivial, and the algebra structure is untwisted. See the next remark.

\begin{remark}
Using the projective covers $P(0)$ and $P(l-2)$ of the simples $L(0)$ and $L(l-2)$ in ${\rm Rep}(\ul_0(\s_2))$, we may identify the principle block ${\rm Rep}(\ul_0)$ with module over the endomorphism algebra
$
E:=\End_{\ul}(P(0)\oplus P(l-2))
$
The algebra can readily be computed to be equivalent to the quiver algebra with two vertices
\[
\xymatrix{
\bullet \ar@/_1pc/[rr]|{a_2}\ar@/_0.3pc/[rr]|{a_1} && \ar@/_1pc/[ll]|{b_2} \ar@/_0.3pc/[ll]|{b_1} \bullet
}
\]
modulo relations
\[
a_ib_j=0=b_ja_i~(i\neq j), \quad  a_1b_1=a_2b_2, \quad b_1a_1=b_2a_2.
\]
One may readily check that the algebra is Koszul, and thus there is a hidden grading on ${\rm Rep}(\ul_0(\s_2))$ whose graded representation theory is equivalent to the graded representation theory of $E$. One can check that the $\C^*$-grading above in Table \ref{table-sl2} agrees with the Koszul grading.
\end{remark}

\begin{corollary}   \label{HH_sl2_alg}
The algebra structure of $\mHH^{\bullet}(\ul(\s_2))$ is determined by
\[
\mu_i^2=\nu_i^2=\nu_i\mu_i+\mu_i\nu_i=0, \quad \quad e_i\mu_i=\mu_ie_i=\mu_i, \quad \quad e_i\nu_i=\nu_ie_i=\nu_i, \quad \quad
\]
\[
x_i\mu_i=\mu_ix_i=y_i\mu_i=\mu_iy_i=0, \quad \quad x_i\nu_i=\nu_ix_i=y_i\nu_i=\nu_iy_i=0,
\]
\[
\mu_i\nu_i=-\nu_i\mu_i=x_i-y_i.
\]
\end{corollary}

\begin{proof} 
The associative algebra structure on $\mHH(\ul_0)\cong \mHH_{\C^*}(T\Nt)$, as given in the Theorem  \ref{BL}, is determined by the exterior algebra structure on $\Nt$ up to a Todd class twist by \cite{CaVdB}. For $\Nt=T^*\PS$, the Todd class is trivial, and we may compute the algebra structure directly.

Cover $\PS$ by affine lines 
\[
\PS=\mathrm{Spec}(\C[z])\cup \mathrm{Spec}(\C[w]),
\]
where $z=1/w$ over the common intersection. Then $\Nt=T^*\PS$ is covered by the affine planes
\[
T^* \PS = \mathrm{Spec}(\C[z, p_z]) \cup \mathrm{Spec}(\C[w,p_w]),
\]
with it understood that $p_z$ is dual to $dz$ and $p_w$ is dual to $dw$.

In the first affine chart, we may identify $\mH^0(\Nt, T\Nt)$ with the vector space
\[
\C\left\langle p_z\frac{\partial}{\partial p_z}, \frac{\partial}{\partial z}, z\frac{\partial}{\partial z}, z^2\frac{\partial}{\partial z} \right\rangle
\]
We note that $\frac{\partial}{\partial p_z}\wedge \frac{\partial}{\partial z}$ is the restriction of the Poisson bivector field $\tau$ in the coordinate chart.

Identify
\[
\C\left\langle p_z\frac{\partial}{\partial p_z}\right\rangle\cong V(0)\otimes \mu_i, \quad \quad \quad
\C \left\langle \frac{\partial}{\partial z}, z\frac{\partial}{\partial z}, z^2\frac{\partial}{\partial z}\right\rangle\cong V(2)\otimes \nu_i.
\]
The natural map
\[
\mH^0(\Nt,T\Nt)^0\otimes \mHH^0(\Nt, T\Nt)^0\stackrel{\wedge}{\lra} \mH^0(\Nt, \wedge^2T\Nt)^0\cong V(2) (x_i-y_i)
\]
factors through
\[
(V(0)\otimes \mu_i)\otimes (V(2)\otimes \nu_i)) \cong V(2) \tau.
\]
It follows that we have to set $\mu_i\nu_i=-\nu_i\mu_i$ to be a degree-$2$ central element, which we may choose to be $x_i-y_i$ up to isomorphism. The result follows.

\end{proof} 

\paragraph{Modular group action.}
In \cite{Ker}, it is established that there is the following decompostion of $\zl(\s_2)=\zl(\ul_q(\s_2))$ as a module over the modular group.

\begin{theorem}
At a primitive root of unity $q$ of order $l$, the projective modular group $SL_2(\Z)$ action on the center of $\ul_q(\s_2)$ decomposes into
\[
\zl(\s_2)\cong P_{\frac{l+1}{2}} \oplus \C^2\otimes V_{\frac{l-1}{2}},
\]
where $P_{\frac{l+1}{2}} $ is an $\frac{l+1}{2}$-dimensional module of $SL_2(\Z)$, $\C^2$ stands for the standard matrix representation  and $V_{\frac{l-1}{2}}$ is an $\frac{l-1}{2}$-dimensional representation of $SL_2(\Z)$.
\end{theorem}

\begin{corollary}
Inside the adjoint representation, there is a decomposition 
\[
P_{\frac{l+1}{2}} \cong \C \left\langle x_i+y_i, e_{\rm st}\right\rangle_{i=1}^\frac{l-1}{2},\quad \quad
\C^2\otimes V_{\frac{l-1}{2}} \cong  \C \left\langle x_i-y_i, e_{i}\right\rangle_{i=1}^\frac{l-1}{2}.
\]
\end{corollary}
\begin{proof}
Since the modular group action preserves the module structure of $\ul_q(\s_2)_\ad$ (Theorem \ref{thm-modular-group-action}), it preserves the sum of trivial submodules contained in projective summands as well as trivial summands of $\ul_q(\s_2)_\ad$. The result now follows from counting dimensions from the  Kerler's Theorem and Theorem \ref{thm-sl2-central-elements-in-adrep}.
\end{proof}

\begin{corollary}  \label{HH-sl2-modular}
The even part of the Hochschild cohomology decomposes, as a projective $SL_2(\Z)$-module, into
\[
\mHH^{2\bullet}(\ul_q(\s_2))\cong P_{\frac{l+1}{2}} \oplus \C^2\otimes V_{\frac{l-1}{2}} \otimes \C[\Nc].
\]
\end{corollary}

It remains an interesting question to investigate the projective $SL_2(\Z)$-action on $\C[\Nc]$ so as to determine the exact modular group structure in the above results.

%%%%%%%%%%%%%%%%%%%%%%%%%%%%%%%%%%%%%%%%%%%%%%%%%
%%%%%%%%%%%%%%%%% REFERENCES %%%%%%%%%%%%%%%%%%%%
%%%%%%%%%%%%%%%%%%%%%%%%%%%%%%%%%%%%%%%%%%%%%%%%%

\addcontentsline{toc}{section}{References}

\bibliographystyle{alpha}
\bibliography{qy-bib}

%%%%%%%%%%%%%%%%%%%%%%%%%%%%%%%%%%%%%%%%%%%%%%%%%
%%%%%%%%%%%%%%%%%%%%%%%%%%%%%%%%%%%%%%%%%%%%%%%%%

%
% ====================================================================

\vspace{0.1in}

\noindent A.~L.: { \sl \small Department of Mathematics, \'{E}cole Polytechnique F\'{e}d\'{e}rale de Lausanne,  SB MATH TAN, MA C3 535 (B\^atiment MA), Station 8, CH-1015 Lausanne, Switzerland} \newline \noindent {\tt \small email: anna.lachowska@epfl.ch}

\vspace{0.1in}

\noindent Y.~Q.: { \sl \small Department of Mathematics, University of Virginia, Charlottesville, VA 22904, USA} \newline \noindent {\tt \small email: yq2dw@virginia.edu}

% ==============================================================================
%

\end{document}